\documentclass{amsart}
\usepackage{graphicx}
\usepackage{rotating}
\graphicspath{{graphics/}}
\usepackage{amscd,amssymb,url}
\usepackage[frame,curve,rotate,arrow,matrix,poly]{xy}

\newtheorem{thm}{Theorem}[section]
\newtheorem{metathm}[thm]{Metatheorem}
\newtheorem{lem}[thm]{Lemma}

\theoremstyle{definition}
\newtheorem{df}[thm]{Definition}
\newtheorem{ex}[thm]{Example}

\theoremstyle{remark}

\newtheorem*{rem}{Remark}
\newtheorem{ack}{Acknowledgments}

\def\CC{\mathcal C}
\def\DD{\mathcal D}
\def\F{\mathcal F}
\def\GG{\mathcal G}
\def\G{\mathfrak G}
\def\gtg{\mathfrak g}
\def\Z{\mathbb Z}
\def\R{\mathbb R}
\def\TT{\mathcal T}
\def\U{\mathcal U}

\DeclareMathOperator{\Ad}{Ad}
\DeclareMathOperator{\AKSZ}{AKSZ}
\DeclareMathOperator{\cl}{cl}
\DeclareMathOperator{\CS}{CS}
\DeclareMathOperator{\D}{D}
\DeclareMathOperator{\DR}{DR}
\DeclareMathOperator{\id}{id}
\DeclareMathOperator{\op}{op}
\DeclareMathOperator{\bbid}{\mathbf{id}}
\DeclareMathOperator{\Cob}{\mathbf{Cob}_n}
\DeclareMathOperator{\GW}{GW}
\DeclareMathOperator{\Hom}{Hom}
\DeclareMathOperator{\mor}{Mor}
\DeclareMathOperator{\Ob}{Ob}
\DeclareMathOperator{\Map}{Map}
\DeclareMathOperator{\Path}{Path}
\DeclareMathOperator{\RW}{RW}
\DeclareMathOperator{\Tr}{Tr}
\DeclareMathOperator{\YM}{YM}

\newcommand{\abs}[1]{{\lvert#1\rvert}}
\newcommand{\Cat}{\mathbf{Cat}}
\newcommand{\del}{\partial}
\newcommand{\nCat}{\text{$n$-$\mathbf{Cat}$}}
\newcommand{\Mor}{\mathbf{Mor}}
\newcommand{\Set}{\mathbf{Set}}
\newcommand{\Th}[1]{\mathbf{Th}(#1)}
\newcommand{\eps}{\varepsilon}

\begin{document}

\title[A Higher Category of Cobordisms and TQFT]{A Higher Category of
  Cobordisms and Topological Quantum Field Theory}
\author[M.~Feshbach]{Mark Feshbach}
\author[A.~A. Voronov]{Alexander A. Voronov}
\address {School of Mathematics\\University of Minnesota\\
  Minneapolis, MN 55455\\USA\\ and IPMU\\University of Tokyo\\Kashiwa,
  277-8583\\Japan}
\email{voronov@umn.edu}

\thanks{Partially supported by NSF grant DMS-0805785, JSPS grant
  L-10514, and World Premier International Research Center Initiative,
  MEXT, Japan}

\date{August 25, 2011}

\begin{abstract}
  The goal of this work is to describe a categorical formalism for
  (Extended) Topological Quantum Field Theories (TQFTs) and present
  them as functors from a suitable category of cobordisms with corners
  to a linear category, generalizing 2d open-closed TQFTs to higher
  dimensions. The approach is based on the notion of an $n$-fold
  category by C.~Ehresmann, weakened in the spirit of monoidal
  categories (associators, interchangers, Mac Lane's pentagons and
  hexagons), in contrast with the simplicial (weak Kan and complete
  Segal) approach of Jacob Lurie. We show how different Topological
  Quantum Field Theories, such as gauge, Chern-Simons, Yang-Mills,
  WZW, Seiberg-Witten, Rozansky-Witten, and AKSZ theories, as well as
  sigma model, may be described as functors from the pseudo $n$-fold
  category of cobordisms to a pseudo $n$-fold category of sets.
\end{abstract}

\maketitle


\section*{Introduction}

In recent years, there has been an increased interest to higher
categories or $n$-categories
\cite{leinster,cheng-lauda,lurie:topoi}. One of the motivating ideas
for that development was Grothendieck's idea of the fundamental
$n$-groupoid of a topological space. Since any kind of associativity
in the fundamental $n$-groupoid is expected to be satisfied only up to
homotopy, similar to composition of based loops in a space, so-called
``weak'' $n$-categories have been of primary interest in the
higher-category community.  In this paper, we focus on a close
relative of the fundamental $n$-groupoid, the category of cobordisms
with corners. We introduce the notion of a pseudo $n$-fold category,
which is a weak version of a classical strict notion by C.~Ehresmann
\cite{ehresmann:1963,ehresmann2}, and show that cobordisms with
corners naturally possess the structure of a pseudo $n$-fold category,
see Theorem~\ref{cob-thm}. This weak version is similar to the
familiar weakness in monoidal coategories: associators and
interchangers are part of the structure, and coherence axioms include
pentagons and hexagons. One of the main results of the paper is a
Regular Coherence theorem, Theorem~\ref{coherence} of the Appendinx,
indicates that our set of coherence axioms is complete, that is to
say, no matter what coherence maps one chooses to go from one way of
composing morphisms to another, the resulting coherence maps will be
equal. This theorem is analogous to Mac Lane's coherence theorems for
monoidal and symmetric monoidal categories and Joyal-Street's
coherence theorem for braided monoidal categories and possibly a
generalization theoreof.

Another goal of our paper is to present (extended) Topological Quantum
Field Theories (TQFTs), including gauge, Chern-Simons, Yang-Mills, and
Seiberg-Witten theories, and sigma model, as (contravariant, lax)
monoidal functors from the monoidal pseudo $n$-fold category of
cobordisms with corners to a certain monoidal $n$-fold category of
spans of sets, which are set-theoretic counterparts of cobordisms with
corners. The problem of presenting TQFTs as such functors was set up
by Baez and Dolan \cite{baez-dolan}, who formulated it as part of the
extended TQFT hypothesis, which also anticipates that extended TQFTs
are classified by their values on the one-point space. J.~Lurie
\cite{lurie:tft} described all the ingredients of such formalization
and proved the extended TQFT hypothesis. S.~Morrison and K.~Walker
\cite{morrison-walker} gave a very interesting definition of a weak
$n$-category and based their (derived) TQFTs on $n$-cobordisms in
which all inputs, boundaries, creases, and corners were essentially
morphed into one boundary. On a more physical side, A.~Kapustin
\cite{kapustin} indicated what one should expect from higher
categories to guarantee that actual models of quantum field theory are
indeed described as extended TQFT functors. Our main result here is a
demonstration that numerous physical models, from gauge theory to
sigma model, may indeed be described as extended TQFTs in our
formalism, see Metatheorem~\ref{meta}.

We have chosen an approach to higher categories with concrete
composition laws and coherences given by concrete associators and
interchangers. Another approach to higher categories, invented by
Boardman and Vogt \cite{boardman-vogt:structures} and developed in
subsequent works of Joyal \cite{joyal:quasi2002} and Lurie
\cite{lurie:topoi}, uses ``fuzzy'' compositions and coherences and has
proven to be very successful. We could have followed that path as well
and worked with the notion of an $n$-fold quasi-category, which we
introduce in Section~\ref{lurie}, but decided to develop an approach
hinted on and attempted in the works of Baez and Dolan
\cite{baez-dolan}, Verity, Morton, and Grandis, to name a few. This
approach has all the advantages of a hands-on method: since there
exist canonical, familiar compositions in such examples as cospans of
sets or cobordisms, it is tempting to have a theory based on them. We
have found it amazing that this theory has not yet been developed to a
fully coherent theory, and this is one of the goals of the present
paper.

In this paper, we deliberately avoid the question of Baez-Dolan's
cobordism (or tangle) hypothesis \cite{baez-dolan}, so as to reduce
categorical considerations to a minimum and concentrate on the
physical examples of Section~\ref{TQFT}. It would be interesting to
introduce duals \`a la Baez and Dolan in our context and verify the
cobordism hypothesis.

If the reader intends to skip the present paragraph, it may appear to
him that we must run into all sorts of set-theoretic difficulties in
the paper. Those may be resolved by using standard Grothendieck's
trick: assume the existence of two universes, one being an element of
the other: $\U \in \U'$. Then by default, when we talk about sets,
spaces, etc., we will assume those to be $\U$-small. Then the
categories of such will automatically be $\U'$-small. These will be
our default assumptions throughout the paper.

\begin{ack}
  We are grateful to John Baez, Michael Batanin, Ezra Getzler, Marco
  Grandis, Mark Hovey, Anton Kapustin, Tyler Lawson, Andrei Losev,
  Peter May, Bill Messing, Jae-Suk Park, Michael Shulman, and Jim
  Stasheff for stimulating discussions. While this paper was being
  written, the first author, Mark Feshbach, passed away suddenly and
  prematurely. The second author would like to acknowledge that
  without Mark, who brought in his depth and mathematical culture to
  the project and supported it with persistent interest, this project
  would have never been realized in present form, if realized at
  all. The second author also thanks IPMU for hospitality during the
  final stage of writing the paper.
\end{ack}

\section{Strict $n$-Fold Categories}

In this section we give a brief account of (strict) $n$-fold
categories of C.~Ehresmann \cite{ehresmann:1963,ehresmann2}, including
three equivalent definitions of an $n$-fold category. Details may be
found in \cite{fiore-paoli} and \cite{pfeiffer}.

There is an algebraic gadget which describes the usual notion of a
category. This gadget is called the \emph{theory of categories},
denoted $\Th{\Cat}$. It is an \emph{essentially algebraic theory},
also known as a \emph{finitely complete category}, which is just a
category with finite limits, or equivalently, pullbacks and a terminal
object. The statement that $\Th{\Cat}$ describes the notion of a
category means that a ($\U'$-small, see the last paragraph of the
Introduction) category may be defined as an \emph{algebra} over
$\Th{\Cat}$, which means a \emph{left exact}, i.e., finite-limit
preserving, functor $\Th{\Cat} \to \Set$, where $\Set$ is the category
of ($\U'$-small) sets. The theory $\Th{\Cat}$ of categories is
generated (as a finitely complete category) by two objects $\Ob$ and
$\mor$ and four morphisms $s, t: \mor \to \Ob$, $\id: \Ob \to \mor$,
and $\circ: \mor {}_t \times_s \mor \to \mor$, with a number of
relations imposed, including $\circ (\circ \times \id) = \circ (\id
\times \circ)$, which encodes the associativity of morphism
composition in a category, relations encoding what the source and the
target of identity morphisms and compositions of morphisms should be,
and those encoding the behavior of identity morphisms as units under
composition, see details in \cite{pfeiffer}.

\begin{df}
  An $n$-\emph{fold category} is a \emph{left exact} (i.e., limit
  preserving) functor
\[
F: \Th{\Cat}^n \to \Set.
\]
\end{df}
In particular, a $0$-fold category is a set and a $1$-fold category is
just a category. $2$-Fold categories are also known as \emph{double
  categories}.

\begin{rem}
  Note that by definition, the (essentially algebraic) theory of
  $n$-fold categories is nothing but the $n$th Cartesian power of the
  theory $\Th{\Cat}$ of categories.
\end{rem}

The following, equivalent definition of an $n$-fold category is
inductive. Set $\Cat_0 := \Set$ and assume we have defined the
category $\Cat_{n-1}$ of $(n-1)$-fold categories and proven it is
finitely complete.
\begin{df}
\label{inductive}
An $n$-\emph{fold category} is a left exact functor
\[
F: \Th{\Cat} \to \Cat_{n-1}.
\]
Equivalently, an $n$-fold category is an \emph{internal category} in
$\Cat_{n-1}$.
\end{df}

Another equivalent definition of an $n$-fold category may be given by
specifying sets of 0-, 1-, \dots, and $n$-morphisms and maps defining
relationships between these sets.
\begin{df}
\label{naive}
An $n$-\emph{fold category} consists of the following data satisfying
the following axioms.
\begin{description}
\item[Morphisms]
\begin{itemize}
\item A set $X$ of 0-morphisms (also called the \emph{objects});
\item For each $i$, $1 \le i \le n$, a set $X_i$ of 1-morphisms in the
  direction labeled $i$;

\dots

\item For each $k$-combination\footnote{Here a \emph{combination}
    means an (unordered) subset.} $\{i_1, \dots, i_k\}$ of numbers
  between 1 and $n$, a set $X_{i_1 \dots i_k}$ of $k$-morphisms in the
  directions $i_1, \dots, i_k$, i.e., in the direction of the
  ``$x_{i_1} \dots x_{i_k}$ plane;''

\dots

\item A set $X_{1 \dots n}$ of $n$-morphisms (in all existing
  directions, labeled $1, \dots, n$).
\end{itemize}
Let $kX$ denote the set of all $k$-morphisms.

\item[Sources, targets, identities] For each $k$, $1 \le k \le n$,
  (multi)maps $\mathbf{s}: kX \to (k-1)X$ for sources, $\mathbf{t}: kX
  \to (k-1)X$ for targets, and $\bbid: (k-1)X \to kX$ for
  identities. Here mutlimaps $\mathbf{s}$, $\mathbf{t}$, and $\bbid$
  denote collections of maps $s_{i_k}, t_{i_k}: X_{i_1 \dots i_k} \to
  X_{i_1 \dots i_{k-1}}$, and $\id_{i_k}: X_{i_1 \dots i_{k-1}} \to
  X_{i_1 \dots i_k}$ for each combination $\{i_1, \dots, i_k\}$ of $k$
  numbers between 1 and $n$.

\item[Compositions] For each combination $\{i_1, \dots, i_k\}$ of $k$
  numbers between 1 and $n$, maps
\[
\circ_{i_k}: X_{i_1 \dots i_k} \times_{X_{i_1 \dots i_{k-1}}} X_{i_1
  \dots i_k} \to X_{i_1 \dots i_k}
\]
of composition in direction $i_k$. Here $X_{i_1 \dots i_k}
\times_{X_{i_1 \dots i_{k-1}}} X_{i_1 \dots i_k}$ is determined by a
pullback diagram of sets, as follows:
\[
\begin{CD}
X_{i_1 \dots i_k} \times_{X_{i_1 \dots i_{k-1}}} X_{i_1 \dots i_k} @>{\pi_1}>> X_{i_1 \dots i_k}\\
@V{\pi_2}VV @VV{t_{i_k}}V\\
X_{i_1 \dots i_k} @>{s_{i_k}}>> X_{i_1 \dots i_{k-1}}
\end{CD}.
\]

\item[Axioms] Composition must be \emph{associative}:
\[
(x \circ_i y) \circ_i z = x \circ_i (y \circ_i z),
\]
identities must behave as such under composition:
\begin{gather*}
\id_i s_i (x) \circ_i x = x =  x \circ_i \id_i t_i (x),\\
  \id_i (x \circ_j y) = \id_i (x) \circ_j \id_i(y) \qquad \text{for $i
    \ne j$},
\end{gather*}
and commute with each other and sources and targets:
\begin{gather*}
  \id_i \id_j = \id_j \id_i \qquad \text{for $i \ne j$},\\
  s_i \id_i (x) = t_i \id_i (x) = x,\\
  s_i \id_j (x) = \id_j s_i (x), \qquad t_i \id_j (x) = \id_j t_i (x),
  \qquad \text{for $i \ne j$};
\end{gather*}
compositions in different directions $i \ne j$ must commute with each
other (the \emph{interchange law}):
\[
(x \circ_i y) \circ_j (x' \circ_i y') = (x \circ_j x') \circ_i (y \circ_j y');
\]
compositions must be compatible with the source and target maps in the
following sense:
\begin{gather*}
  s_i (x \circ_i y) = s_i (x), \qquad t_i (x \circ_i y) = t_i (y),\\
  s_i (x \circ_j y) = s_i (x) \circ_j s_i (y), \quad t_i (x \circ_j
  y) = t_i (x) \circ_j t_i(y), \qquad \text{for $i \ne j$};
\end{gather*}
and sources and targets in different directions $i \ne j$ must commute
with each other:
\[
s_i s_j (x) = s_j s_i (x), \qquad t_i s_j (x) = s_j t_i (x), \qquad
t_i t_j (x) = t_j t_i (x).
\]
\end{description}
\end{df}

\begin{rem}
\begin{sloppypar}
  There is still other, equivalent single-set definition, which
  describes the set of $n$-morphisms and treats lower-dimensional
  morphisms as degenerate $n$-morphisms, see \cite{ehresmann2}.
\end{sloppypar}
\end{rem}

\begin{rem}
  Every strict $n$-category has the natural structure of an $n$-fold
  category, moreover, we have a functor $\nCat \to \Cat_n$, where
  $\nCat$ is the category of strict $n$-categories, cf.\ \cite[Section
  II.5]{maclane}. In terms of Definition~\ref{naive}, if we require
  that $k$-morphisms in every plane but the $x_1 x_2 \dots x_k$ plane
  be identities, we will get a definition of a strict
  $n$-category. This yields a morphism of theories $\Th{\Cat_n} \to
  \Th{\nCat}$.
\end{rem}

One advantage of Definition~\ref{naive} is that it allows us to talk
about \emph{$\infty$-fold $($or $\omega$-fold$)$ categories} directly,
without passing to the limit $n \to \infty$: to define an
$\infty$-fold category, we will just set $n = \infty$ in
Definition~\ref{naive}, while the index $i$ for directions will still
be finite, $1 \le i < \infty$, and $k$-morphisms will also span in a
finite number $k$, $0 \le k < \infty$, of directions $\{i_1, \dots,
i_k\}$. We will not use $\infty$-fold categories in this paper, except
in the following example.

\begin{ex}[The fundamental $\omega$-fold groupoid of a topological
  space] We have learned this version of Grothendieck's fundamental
  $\omega$-groupoid from U.~Tillmann \cite{tillmann}, see also
  R.~Brown \cite{brown}. For a topological space $X$, let the set
  $X_{i_1 \dots i_k}$ of $k$-morphisms be the set of pairs $(I,f)$,
  where $I = [a_1,b_1] \times [a_2, b_2] \times \dots$, where $a_1 \le
  b_1$, $a_2 \le b_2$, etc., so that $a_i = b_i$ whenever $i \not \in
  \{i_1, \dots, i_k\}$ and $a_i = b_i = 0$ for $i >> 0$, is a
  coordinate rectangular solid in $\R^\infty$ of dimension at most $k$
  and $f: I \to X$ is a continuous map. Composition of $k$-morphisms
  in the $i$th direction is defined, if and only if the rectangular
  solids fit together along their $i$th faces in $\R^\infty$ to form a
  larger rectangular solid and the restrictions of the maps from the
  solids to $X$ to the two faces are equal. In this case, the two maps
  glue to a map from their union to $X$ by continuity. If we fix $n$
  and use $\R^n$ instead of $\R^\infty$ in the above, we will obtain a
  $n$-truncation of that construction, which will be an $n$-fold
  category. For instance, for $n=1$ we get the (Moore) path category
  of the topological space.
\end{ex}

We will mention another example of a strict $n$-fold category, when we
will discuss embedded cobordisms with corners below.

\section{The Nerve of an $n$-Fold Category}

The \emph{nerve of a $($strict$)$ $n$-fold category} $X_\bullet$ is an
$n$-fold simplicial set $\mathbf{X}: (\mathbf{\Delta}^{\op})^n \to
\Set$, where $\mathbf{\Delta}$ is the standard simplex category. The
nerve may be constructed as the iterated nerve of a usual category,
using the fact that an $n$-fold category is a category object in the
category of $(n-1)$-fold categories. Thus, in order to construct the
simplicial nerve of an $n$-fold category, we first construct the nerve
of it as a category object: for $k > 0$, a $k$-simplex in this
simplicial object is a composable sequence of $k$ morphisms in our
category object, i.e., functors between $(n-1)$-categories; for $k =
0$, this is just an object, i.e., an $(n-1)$-category. Each
$(n-1)$-category gives rise to an $(n-1)$-fold simplicial set by
induction, and the above nerve, which is a simplicial object in the
category of $(n-1)$-fold categories, turns into a simplicial object in
the category of $(n-1)$-fold simplicial sets.

The following theorem gives a test for an $n$-fold simplicial set to
be the nerve of an $n$-fold category via an inner-horn filling
condition. Recall that for the standard simplex $\Delta^k$, considered
as a simplicial set, and any $0 \le j \le k$, the $j$th \emph{horn}
$\Lambda_j^k \subset \Delta^k$ is obtained from $\Delta^k$ by deleting
the interior and the face opposite the $j$th vertex.

\begin{thm}
\label{nerve}
An $n$-fold simplicial set $\mathbf{X}$ is the nerve of a strict
$n$-fold category, if and only if it satisfies the \emph{unique inner
  horn-filling condition:} for each sequence $\sigma_1, \dots,
\sigma_n$, where each $\sigma_i$ is either a simplex $\Delta^{k_i}$,
$k_i \ge 0$, or an inner horn $\Lambda_j^{k_i} \subset \Delta^{k_i}$,
$0 < j < k_i$, any \emph{inner multihorn} $(\sigma_1, \dots, \sigma_k)
\to \mathbf{X}$ may be completed to a multisimplex $(\Delta^{k_1},
\dots, \Delta^{k_n}) \to \mathbf{X}$.
\end{thm}

\begin{proof}
  This theorem follows by iterating the $n=1$ statement, known as
  Boardman-Vogt's theorem \cite{boardman-vogt:structures}.
\end{proof}

The geometric realization of the nerve of an $n$-fold category
$X_\bullet$ is called the \emph{classifying space} of it, denoted $B
X_\bullet$.

\begin{ex}
  The classifying space of the fundamental $n$-fold groupoid of a
  topological space $X$ is homotopy equivalent to the space. For
  $n=1$, this follows from the observation that the path space of a
  topological space $X$ is homotopy equivalent to $X$. For higher
  $n$'s, the result follows by iteration.
\end{ex}

\section{Pseudo $n$-Fold Categories}
\label{pseudo}

For abstract cobordisms with corners, we will need to introduce
``controlled weakness'' into the notion of a strict $n$-fold
category. We will call the corresponding notion a pseudo $n$-fold
category, which will be a close relative of Marco Grandis' notion of a
\emph{symmetric weak $(n+1)$-cubical category}, see
\cite{grandis-I}. The principal difference is that we add a second
hexagon axiom, which makes the whole notion coherent, see Theorems
\ref{existence} and \ref{coherence}. Technical differences include our
omitting the symmetric structure and using strict units.

\begin{df}
  A \emph{pseudo $n$-fold category} is a certain \emph{weak model} of
  an $n$-fold category in $\Cat$. This means that a pseudo $n$-fold
  category has the same structure of morphisms $k\Mor$ for $k =0, 1,
  \dots, n$, sources and targets $\mathbf{s}, \mathbf{t}: k\Mor \to
  (k-1)\Mor$, identities $\bbid: (k-1)\Mor \to k\Mor$, and
  compositions $\circ: k\Mor \times_{(k-1)\Mor} k\Mor \to k\Mor$, $1
  \le k \le n$, as a (strict) $n$-fold category, see
  Definition~\ref{naive}, except that now $k\Mor$ is a category,
  rather than a set, and $\mathbf{s},\mathbf{t}$, $\bbid$, and
  compositions are functors. Thus, a pseudo $n$-fold category has
  extra $(k+1)$-morphisms in a ``transversal'' direction, the
  morphisms of the categories $k\Mor$, $k = 0, 1, \dots, n$.  Some of
  the axioms must be satisfied weakly, that is, up to natural
  isomorphism of functors. These natural isomorphisms are part of the
  structure, satisfying their own coherence axioms. We will go over
  this extra data and the axioms one by one in detail below. Letters
  $x$, $y$, $z$, etc., will be placeholder for objects in categories
  $k\Mor$.

\noindent
\textbf{Associators:} There must be natural isomorphisms
\begin{eqnarray*}
\alpha_i: \circ_i (\circ_i \times \id) & \to & \circ_i (\id \times
\circ_i),\\
(x \circ_i y) \circ_i z & \mapsto & x \circ_i (y \circ_i z),
\end{eqnarray*}
of functors
\[
k\Mor \times_{(k-1)\Mor} k\Mor \times_{(k-1)\Mor} k\Mor \to k\Mor
\]
for each $k = 1, \dots, n$ and $i = 1, \dots n$. (To be more precise,
$\alpha_i$ is a collection of isomorphisms depending on several other
indices.)

\noindent
\begin{sloppypar}
\textbf{Interchangers:} For each $k = 1, \dots, n$ and pair of distinct
  numbers $i$ and $j$ between 1 and $n$, there must be natural
  isomorphisms
\begin{eqnarray*}
  \beta_{ij}: \circ_j (\circ_i \times \circ_i) & \to & \circ_i (\circ_j \times
  \circ_j),\\
  (x \circ_i y) \circ_j (x' \circ_i y') & \mapsto & (x \circ_j x') \circ_i (y \circ_j y'),
\end{eqnarray*}
of functors
\[
k\Mor^4_{(k-1)\Mor^4} \to k\Mor,
\]
where $k\Mor^4_{(k-1)\Mor^4}$ is the category pullback describing
quadruples of $k$-morphisms composable horizontally (i.e., in the
$i$th direction) and vertically (i.e., in the $j$th direction), as per
the following square:
\[
\xymatrix{
\bullet \ar[r] \ar[d] \ar @{=>}[]-/ul 19pt/; [dr]+/ul 19pt/ ^x & \bullet \ar[r] \ar[d] \ar @{ =>} []-/ul 19pt/; [dr]
+/ul 19pt/^y  & \bullet \ar[d]\\
\bullet \ar[r] \ar[d] \ar @{ =>} []-/ul 19pt/; [dr]+/ul 19pt/^{x'} & \bullet \ar[r] \ar[d] \ar @{ =>} []-/ul 19pt/; 
[dr]+/ul 19pt/^{y'} & \bullet \ar[d]\\
\bullet \ar[r] &\bullet \ar[r] & \bullet}
\]
\end{sloppypar}

\noindent
\textbf{Axioms:} First of all, we must have
\[
\beta_{ji} = \beta_{ij}^{-1}.
\]
Next, the \emph{coherence axioms} below for the associators and
interchagers must be satisfied. The diagrams involved are diagrams of
natural transformations between morphism categories $k\Mor$, and
commutativity of diagrams means equality of the corresponding natural
transformations.

The following \emph{pentagon} diagram must commute for each $i$:
\def\vertices{\ifcase\xypolynode\or x\circ_i (y \circ_i (z \circ_i u))
  \or (x\circ_i y) \circ_i (z \circ_i u) \or ((x\circ_i y) \circ_i z)
  \circ_i u \or \hspace{-.5in} (x\circ_i (y \circ_i z)) \circ_i u \or
  x\circ_i ((y \circ_i z) \circ_i u) \hspace{-.5in} \fi}
\[
\xy/r5pc/:
 {\xypolygon5{~><{@{}}~*{\vertices}}}
{\ar^\alpha "2";"1"}
{\ar@<2pt>^-{\alpha} "3";"2" +/dl 14pt/ +/l 5pt/}
{\ar@<-7pt>_-{\alpha \circ_i \id} "3";"4" +/u 10pt/ +/l 3pt/}
{\ar_\alpha "4"-/l 13pt/;"5" !/l 33pt/
}
{\ar@<3ex>_-{\id \circ_i \alpha} "5" +/ 4pt/;"1" -/ 14pt/}
\endxy
\]
This axiom, in other words, requires that the two ways of moving
parentheses from one extreme to the other in a sequence
\[
\begin{CD}
\bullet @>x>> \bullet @>y>> \bullet @>z>> \bullet @>u>> \bullet 
\end{CD}
\]
of composable $k$-morphisms, using associators, be equal.

The following \emph{hexagon} diagram must commute for each pair of
distinct $i$ and $j$:
\def\vertices{\ifcase\xypolynode
\or (x \circ_j x') \circ_i ((y \circ_j y') \circ_i (z \circ_j z')) \hspace{-1in}
\or ((x\circ_j x') \circ_i(y \circ_j y')) \circ_i (z \circ_j z') \hspace{-1.5in}
\or ((x\circ_i y) \circ_j (x' \circ_i y')) \circ_i (z \circ_j z') \hspace{1.4in}
\or \hspace{-1.4in} ((x\circ_i y) \circ_i z) \circ_j  ((x' \circ_i y') \circ_i z')
\or \hspace{-1.4in}
(x\circ_i (y \circ_i z)) \circ_j (x' \circ_i (y' \circ_i z')) 
\or 
(x \circ_j x') \circ_i ((y \circ_i z) \circ_j (y' \circ_i z'))  \hspace{-1.5in}
\fi}
\[
\xy/r5pc/:
 {\xypolygon6{~><{@{}}~*{\vertices}}}
{\ar^-\alpha "2"!/r 30pt/ !/d 8pt/ ;"1"!/r 30pt/}
{\ar^-{\beta \circ_i \id} "3"+/r 17pt/;"2" +/l 20pt/
}
{\ar
^-{\beta} "4" !/l 40pt/;"3" !/l 40pt/ !/d 8pt/
}
{\ar_-{\alpha \circ_j \alpha} "4"!/l 40pt/ ;"5" !/l 40pt/ !/u 7pt/}
{\ar
_-{\beta} "5" -/l 6pt/;"6" !/l 30pt/
}
{\ar
_-{\id \circ_i \beta} "6" !/r 10pt/ !/u 8pt/
;"1" !/r 10pt/
}
\endxy
\]
This axiom requires that the two ways of moving parentheses from one
extreme to the other in a sequence of morphisms
\[
\xymatrix{
  \bullet \ar[r] \ar[d] \ar @{=>}[]-/ul 19pt/; [dr]+/ul 19pt/ ^x & \bullet \ar[r] \ar[d] \ar @{ =>} []-/ul 19pt/; [dr]
+/ul 19pt/^y & \bullet \ar[r] \ar[d] \ar @{ =>} []-/ul 19pt/; [dr]+/ul 19pt/^z  & \bullet \ar[d]\\
  \bullet \ar[r] \ar[d] \ar @{ =>} []-/ul 19pt/; [dr]+/ul 19pt/^{x'} & \bullet \ar[r] \ar[d] \ar @{ =>} []-/ul 19pt/; 
[dr]+/ul 19pt/^{y'} & \bullet \ar[r] \ar[d] \ar @{ =>} []-/ul 19pt/; [dr]+/ul 19pt/^{z'} & \bullet \ar[d]\\
  \bullet \ar[r] &\bullet \ar[r] &\bullet \ar[r] & \bullet}
\]
using associators and interchangers, be equal.

There is a \emph{second hexagon} axiom that must be satisfied. For
each triple of distinct $i$, $j$, and $k$, the following diagram must
commute:
\def\objectstyle{\scriptstyle}
\def\vertices{\ifcase\xypolynode \or ((x \circ_k x'') \circ_j
  (x' \circ_k x''')) \circ_i ((y \circ_k y'') \circ_j (y' \circ_k
  y''')) \hspace{-1.1in} \or ((x \circ_j x') \circ_k
  (x'' \circ_j x''')) \circ_i ((y \circ_j y') \circ_k (y'' \circ_j
  y'''))  \hspace{-1.8in} \or \hspace{-1.8in}
   ((x \circ_j x') \circ_i
  (y \circ_j y')) \circ_k ((x'' \circ_j x''') \circ_i (y'' \circ_j
  y'''))  \or
  \hspace{-1.1in}     ((x \circ_i y) \circ_j
  (x' \circ_i y')) \circ_k ((x'' \circ_i y'') \circ_j (x''' \circ_i
  y'''))  \or
 \hspace{-1.8in}  ((x \circ_i y) \circ_k
  (x'' \circ_i y'')) \circ_j ((x' \circ_i y') \circ_k (x''' \circ_i
  y'''))  
  \or  ((x \circ_k x'') \circ_i
  (y \circ_k y'')) \circ_j ((x' \circ_k x''') \circ_i (y' \circ_k
  y'''))
  \hspace{-1.8in} \fi}
\[
\xy/r5pc/:
 {\xypolygon6{~><{@{}}~*{\vertices}}}
{\ar^{\beta_{jk} \circ_i \beta_{jk}} "2"!/r 30pt/ !/d 8pt/ ;"1"!/r 30pt/}
{\ar^-{\beta_{ik}} "3"+/r 17pt/;"2" +/l 23pt/
}
{\ar
^-{\beta_{ij} \circ_k \beta_{ij}} "4" !/l 40pt/;"3" !/l 40pt/ !/d 8pt/
}
{\ar_-{\beta_{jk}} "4"!/l 40pt/ ;"5" !/l 40pt/ !/u 7pt/}
{\ar
_-{\beta_{ik} \circ_j \beta_{ik}} "5" -/l 6pt/;"6" !/l 32pt/
}
{\ar
_-{\beta_{ij}} "6" !/r 10pt/ !/u 8pt/
;"1" !/r 10pt/
}
\endxy
\]
This axiom requires that the two ways of moving parentheses from one
extreme to the other in a sequence of morphisms
\[
\xymatrix@=10pt{ & & \bullet \ar@{-}[rr] & & \bullet \ar@{-}[rr] & &
  \bullet \ar@{-}[dd]
  \\
  & \bullet \ar@{-}[ur] \ar@{-}[rr] & & \bullet \ar@{-}[ur]
  \ar@{-}[rr] & & \bullet \ar@{-}[ur] \ar@{-}[dd]
  \\
  \bullet \ar@{-}[ur] \ar@{-}[rr] \ar@{-}[dd] & & \bullet \ar@{-}[ur]
  \ar@{-}[rr] \ar@{-}[dd] & & \bullet \ar@{-}[ur] \ar@{-}[dd] & &
  \bullet \ar@{-}[dd]
  \\
  & & & & & \bullet \ar@{-}[ur] \ar@{-}[dd]
  \\
  \bullet \ar@{-}[rr] \ar@{-}[dd] & & \bullet \ar@{-}[rr] \ar@{-}[dd]
  & & \bullet \ar@{-}[ur] \ar@{-}[dd] & & \bullet
  \\
  & & & & & \bullet \ar@{-}[ur]
  \\
  \bullet \ar@{-}[rr] & & \bullet \ar@{-}[rr]
  & & \bullet \ar@{-}[ur] }
\]
using interchangers between different pairs of directions, be equal.

Like in a strict $n$-fold category, the following identities, where
$x$, $y$, $z$, etc., now stand for objects or morphisms in categories
$k\Mor$ and $i \ne j$, must hold:
\begin{gather}
\label{units}
\id_i s_i (x) \circ_i x = x =  x \circ_i \id_i t_i (x),\\
\label{units2}
  \id_i (x \circ_j y) = \id_i (x) \circ_j \id_i(y),\\
  \id_i \id_j = \id_j \id_i,\\
  s_i \id_i (x) = t_i \id_i (x) = x,\\
\label{units5}
  s_i \id_j (x) = \id_j s_i (x), \qquad t_i \id_j (x) = \id_j t_i (x),\\
  s_i (x \circ_i y) = s_i (x), \qquad t_i (x \circ_i y) = t_i (y),\\
  s_i (x \circ_j y) = s_i (x) \circ_j s_i (y), \qquad t_i (x \circ_j
  y) = t_i (x) \circ_j t_i(y),\\
  s_i s_j (x) = s_j s_i (x), \qquad t_i s_j (x) = s_j t_i (x), \qquad
  t_i t_j (x) = t_j t_i (x).
\end{gather}

There are extra axioms that require the identity morphisms be coherent
with the associators and the interchangers in the following way:
\begin{eqnarray}
\nonumber \alpha = \id: (\id_i s_i (x) \circ_i x) \circ_i y & \to & 
\id_i s_i(x) \circ_i ( x \circ_i y),\\
\nonumber \alpha = \id: (x \circ_i \id_i t_i (x)) \circ_i y & \to & 
x \circ_i (\id_i s_i (y) \circ_i y),\\
\label{units-coh}
\alpha = \id: (x \circ_i y) \circ_i \id_i t_i (y) & \to & 
x \circ_i (y \circ_i \id_i t_i (y)),\\
\nonumber \beta = \id: (x \circ_i y) \circ_j (\id_j t_j (x) \circ_i \id_j t_j (y))
& \to & (x \circ_j \id_j t_j (x)) \circ_i (y \circ_j \id_j t_j (y)),\\
\nonumber \beta = \id: (\id_j s_j(x) \circ_i \id_j s_j (y)) \circ_j (x \circ_i y)
& \to & (\id_j s_j (x) \circ_j x) \circ_i (\id_j s_j (y) \circ_j y).
\end{eqnarray}
\end{df}
A similar argument to one in the case of monoidal categories shows
that the first and the third equalities $\alpha = \id$ follow from the
second one, but we will not worry here about the independence of our
collection of axioms.

An important aspect in defining the notion of a weak higher category
is whether there are ``enough'' coherence axioms. Such aspects are
usually addressed by coherence theorems, which we will prove
below. Suppose we have a composable combination of
$k$-morphisms. Combinatorially, such a combination may be thought of
as a $k$-dimensional cube in $\R^n$ subdivided into rectangular
blocks, placeholders for the morphisms to be composed. Composing these
morphisms means specifying an order in which the blocks are being
glued. The blocks are not supposed to be moved around in the process.

\begin{thm}[Weak Coherence Theorem]
\label{existence}
Any two functors $(k\Mor)^p_c \to k\Mor$ in a pseudo $n$-fold category
built from morphism composition are related by a natural isomorphism
built from associators $\alpha$ and interchangers $\beta$. Here $p \ge
1$ is an integer and $(k\Mor)^p_c$ denotes the full subcategory in
$k\Mor^p$ of morphisms composable in a diagram $c$ given by a cube of
dimension $k$ made of $p$ blocks.
\end{thm}

\begin{proof}
  We need to compare two functors. We will be talking about their
  values on objects, which are $k$-morphisms in the pseudo $n$-fold
  category, making sure the constructions are natural. This will imply
  the functors are related by a natural transformation.

  Let us use induction on the number $p$ of morphisms being
  composed. The statement is trivial for $p=1$. For any $p \ge 2$,
  consider the last morphism compositions in each of the two
  functors. They must be compositions in either the same direction or
  two different ones.

  If they are compositions in the same direction $i$, then they are
  either the same, $x \circ_i y$, in which case, the induction
  assumption applied to $x$ and $y$ separately completes the argument,
  or different: $x \circ_i y'$ and $y" \circ_i z$. Without loss of
  generality, by the induction hypothesis, we can assume that $Sy' = y
  \circ z$ and $Ty" = x \circ_i y$ for some $y$ and natural
  transformations $S$ and $T$ made out of $\alpha$'s and $\beta$'s, in
  which case for the associator $\alpha_i: (x \circ_i y) \circ_i z \to
  x \circ_i (y \circ_i z)$, we have $\alpha_i (Ty" \circ_i z) = x
  \circ_i Sy'$. Then the composition of $T \circ_i \id$, $\alpha_i$,
  and $\id \circ_i S^{-1}$ takes $y" \circ_i z$ to $x \circ_i y'$.

  If the last compositions in the two functors are in different
  directions, $x \circ_i y$ and $x' \circ_j y'$, we can assume by the
  induction hypothesis that $x = S(x_1 \circ_j x_2)$, $y = T(y_1
  \circ_j y_2)$, and $x' = S' (x_1 \circ_i y_1)$, $y' = T'(x_2 \circ_i
  y_2)$ for some morphisms $x_1$, $x_2$, $y_1$, and $y_2$ and natural
  transformations $S$, $T$, $S'$, and $T'$, built out of associators
  and interchangers. Then the composition of these natural
  transformations and their inverses and an interchanger $\beta_{ij}$
  will move $x' \circ_j y'$ to $x \circ_i y$, as we have $\beta_{ij}
  ((S')^{-1} x' \circ_j (T')^{-1} y') = S^{-1} x \circ_i T^{-1} y$.
\end{proof}

A strong coherence theorem would state that a natural transformation
relating two functors in Theorem~\ref{existence} is unique. We prove
an important particular case of such statement in the Appendix.

\section{Monoidal Versions}

\subsection{Monoidal Pseudo $n$-Fold Categories}

\begin{df}
  A \emph{monoidal pseudo $n$-fold category} is a pseudo $(n+1)$-fold
  category with directions labeled $0, 1, \dots, n$ and such that
  $\Mor_{i_1 \dots i_k}$ is the category with one object
  $\mathbf{1}_{i_1 \dots i_k}$ and one morphism $\id_{\mathbf{1}}$ for
  any combination $\{i_1, \dots, i_k\}$ which does not contain $0$.
\end{df}

\begin{rem}
  In particular, the category $0\Mor$ of ``objects'' must have only
  one object $\mathbf{1}$. As concerns the category $1\Mor =
  \coprod_{i=0}^n \Mor_i$ of 1-morphisms, we have $\Mor_i =
  (\mathbf{1}_{i}, \id_{\mathbf{1}})$ for each $i \ge 1$, while
  $\Mor_0$ may be plentiful. We think of objects in $\Mor_0$ as the
  ``objects'' of our monoidal pseudo $n$-fold category. Each
  nontrivial $(k+1)$-morphism, which is an object in $\Mor_{0 i_1
    \dots i_k}$, has one and the same $\mathbf{1}_{i_1 \dots i_k}$ as
  the source and target in the 0th direction. Thus, any two such
  morphisms $x$ and $y$ may be composed in the 0th direction: we will
  denote this composition $x \otimes y$, and this will be the monoidal
  structure.
\end{rem}

\begin{df}
  A \emph{braided monoidal pseudo $n$-fold category} is a pseudo
  $(n+2)$-fold category with directions labeled $-1, 0, 1, \dots, n$
  and $\Mor_{i_1 \dots i_k}$ being the category with one object
  $\mathbf{1} = \mathbf{1}_{i_1 \dots i_k}$ and one morphism, as long
  as $\{-1, 0\} \not \subset \{i_1, \dots i_k\}$. We also require that
  the interchanger $\beta_{-1,0}: x \circ_0 y = (x \circ_{-1}
  \mathbf{1}) \circ_{0} (\mathbf{1} \circ_{-1} y) \to (x \circ_0
  \mathbf{1}) \circ_{-1} (\mathbf{1} \circ_{0} y) = x \circ_{-1} y$ be
  equal to the identity transformation, cf.\ the discussion of the
  Eckmann-Hilton argument below. Here $x$ and $y$ are any objects of
  $\Mor_{i_1 \dots i_k}$ with $\{-1, 0\} \subset \{i_1, \dots
  i_k\}$. The monoidal structure will be given by $x \otimes y := x
  \circ_0 y$.
\end{df}

\begin{rem}
\begin{sloppypar}
  In particular, the categories $0\Mor$ of ``objects'' and $\Mor_i$,
  $i = -1, 0, \dots, n$, of 1-morphisms will each have only one
  object. Each category $\Mor_{ij}$ will also have only one object,
  except for $\Mor_{-1,0}$.
\end{sloppypar}
\end{rem}

\begin{df}
  A \emph{symmetric monoidal pseudo $n$-fold category} is a pseudo
  $(n+3)$-fold category with directions labeled $-2, -1, 0, 1, \dots,
  n$ and $\Mor_{i_1 \dots i_k}$ being the category with one object
  $\mathbf{1}_{i_1 \dots i_k}$ and one morphism, as long as $\{-2, -1,
  0\} \not \subset \{i_1, \dots i_k\}$. We also require that for $i, j
  \le 0$ the interchangers $\beta_{ij}: (x \circ_{i} \mathbf{1})
  \circ_{j} (\mathbf{1} \circ_{i} y) \to (x \circ_j \mathbf{1})
  \circ_{i} (\mathbf{1} \circ_{j} y)$ be equal to the identity
  transformations, cf.\ the discussion of the Eckmann-Hilton argument
  below. The monoidal structure will again be given by $x \otimes y :=
  x \circ_0 y$, as above.
\end{df}

\subsection{The Eckmann-Hilton Argument}

In this section we would like to intrepret a symmetric monoidal pseudo
$n$-fold category as a pseudo $n$-fold category with a symmetric
monoidal structure, rather than a particular kind of an $(n+3)$-fold
category. Let $\mathbf{1}$ be a unique object (a 0-morphism), then
$\mathbf{1}_{-2,-1,0} = \id_{-2} \id_{-1} \id_0 \mathbf{1}$ will be
the identity 3-morphism, an object in $\Mor_{-2,-1,0}$. We will shift
terminology and interpret the objects in $\Mor_{-2,-1,0} \subset
3\Mor$ as objects of our symmetric monoidal pseudo $n$-fold
category. More generally, objects of $k\Mor$ will be interpreted as
$(k-3)$-morphisms.

Using the interchangers $\beta_{-1,0}$ and $\beta_{0,-1}$ and the fact
that $\mathbf{1} := \mathbf{1}_{-2,-1,0}$ is a unit for compositions
of morphisms in both directions $-1$ and $0$, we can use the
Eckmann-Hilton argument to get the following natural isomorphisms:
\begin{multline*}
  x \circ_0 y = (\mathbf{1} \circ_{-1} x) \circ_0 (y \circ_{-1}
  \mathbf{1}) \xrightarrow{\beta_{-1,0}} (\mathbf{1}
  \circ_0 y) \circ_{-1} (x
  \circ_0 \mathbf{1}) \\
  = y \circ_{-1} x = (y \circ_0 \mathbf{1}) \circ_{-1} (\mathbf{1}
  \circ_{0} x) \xrightarrow{\beta_{0,-1}} (y \circ_{-1}
  \mathbf{1}) \circ_{0} (\mathbf{1} \circ_{-1} x) = y \circ_0 x
\end{multline*}
for any objects $x$ and $y$ of $\Mor_{-2,-1,0}$, as well as any
composable in directions $-1$ and $0$ morphisms in $k\Mor$ for $k \ge
2$. Note that the second natural isomorphism $\beta_{0,-1}$ is
identity by definition, which implies that the two product laws are
equal:
\[
x \circ_0 y = x \circ_{-1} y,
\]
to which we referred as the monoidal structure, or the tensor product:
\[
x \otimes y = x\circ_0 y.
\]
Thereby, the composition of the isomorphisms in the Eckmann-Hilton
argument provides a braiding operator:
\[
\beta: x \otimes y \to y \otimes x.
\]
We can repeat the same argument to show that $x \circ_0 y = x
\circ_{-2} y$. To show that we get only one braiding operator and that
it is symmetric, we can use the second hexagon axiom with the left
middle vertex
\[
((\mathbf{1} \circ_{-2} \mathbf{1}) \circ_{-1} (x \circ_{-2}
\mathbf{1})) \circ_0 ((\mathbf{1} \circ_{-2} y) \circ_{-1} (
\mathbf{1} \circ_{-2} \mathbf{1})).
\]
The second hexagon diagram in this case will look as follows:
\def\vertices{\ifcase\xypolynode \or x \circ_{-2} y \hspace{-.8in} \or x \circ_{-2} y  \hspace{-.5in} \or 
\hspace{-.6in}
  x \circ_{0} y  \or
  \hspace{-.8in}  x \circ_{0} y \or
 \hspace{-.6in}  y \circ_{-1} x  
  \or   y \circ_{-1} x
  \hspace{-.5in} \fi}
\[
\xy/r5pc/:
 {\xypolygon6{~><{@{}}~*{\vertices}}}
{\ar^{\beta_{-1,0} = \id} "2"!/r 30pt/ !/d 8pt/ ;"1"!/r 27pt/}
{\ar^-{\beta_{-2,0} = \id} "3" +/r 5pt/;"2" +/l 12pt/
}
{\ar
^-{\beta_{-2,-1} = \id} "4" !/l 40pt/;"3" !/l 40pt/ !/d 8pt/
}
{\ar_-{\beta_{-1,0}} "4"!/l 40pt/ ;"5" !/l 40pt/ !/u 7pt/}
{\ar
_-{\beta_{-2,0} = \id} "5" -/r 5pt/;"6" !/l 22pt/
}
{\ar
_-{\beta_{-2,-1}} "6" !/r 10pt/ !/u 8pt/
;"1" !/r 6pt/ !/d 7pt/
}
\endxy
\]
The commutativity of the diagram implies that $\beta_{-2,-1}
\beta_{-1,0} = \id$, whence $\beta_{-2,-1} = \beta_{0,-1}$. Since the
same argument works for any permutation $(i,j,k)$ of $(-2,-1,0)$, we
get $\beta_{ij} = \beta_{ik} = \beta_{jk} = \beta_{ji}$. It follows
that all these braiding operators are equal and since $\beta_{ij} =
\beta_{ji}^{-1}$, we also have $\beta^2 = \beta_{ij}^2 = \id$.

\begin{rem}
  Classical monoidal categories may be described as particular
  examples of monoidal pseudo ($1$-fold) categories. For example, a
  braided monoidal category would be a braided monoidal pseudo
  category satisfying the following conditions: the associator
  $\alpha_1$ in the ``spatial'' direction 1 and the interchangers
  $\beta_{i1}$ for $i = -1, 0$ must be equal to identity. We
  anticipate that the classical coherence theorems of Mac Lane
  \cite{maclane:1963b,maclane} for monoidal and symmetric monoidal
  categories and of Joyal-Street \cite{joyal-street:1993,maclane} for
  braided monoidal categories (in the unitary case) follow from our
  Regular Coherence Theorem~\ref{coherence}.
\end{rem}

\section{Functors}
\label{functors}

  A \emph{functor} $F: \CC \to \DD$ between two pseudo $n$-fold
  categories $\CC$ and $\DD$ is a collection of functors $F: k\Mor
  (\CC) \to k\Mor(\DD)$ together with natural transformations
\[
\phi_{x,y}: Fx \circ_i Fy \to F (x \circ_i y),
\]
called the \emph{coherence maps}, commuting with sources, targets, and
identities, and satisfying the following coherence conditions, making
sure the natural transformations are compatible with the identity axioms,
associators, and interchangers:
\begin{gather*}
\phi_{\id_i s_i x,x} = \id_{Fx}, \quad \phi_{x,\id_i t_i x} =
\id_{Fx},
\\
\begin{CD}
  (Fx \circ_i Fy) \circ_i F z @>{\phi_{x,y} \circ_i \id}>> F (x \circ_i y) \circ_i Fz @>{\phi_{x \circ_i y, z}}
>> F((x \circ_i y) \circ_i z)\\
@V\alpha_iVV @. @VV{F(\alpha_i)}V\\
  Fx \circ_i (Fy \circ_i F z) @>{\id \circ_i \phi_{y,z}}>> F x \circ_i
  F(y \circ_i z)  @>{\phi_{x, y  \circ_ i z}}>> F (x \circ_i (y \circ_i z)),
\end{CD}\\
\xymatrix@C=14pt{
  (Fx \circ_i Fy) \circ_j (Fx' \circ_i Fy') 
  \ar[d]_{\beta_{ij}} \ar[r]^-{\phi \circ_j
    \phi} & F (x \circ_i
  y) \circ_j F(x' \circ_i y')  \ar[r]^{\phi} &  F ((x \circ_i y) \circ_j (x' \circ_i y'))
\ar[d]^{F (\beta_{ij})}\\
  (Fx \circ_j Fx') \circ_i (Fy \circ_j Fy') \ar[r]^-{\phi \circ_j
    \phi} &  F (x \circ_j x') \circ_i F(y \circ_j y') \ar[r]^{\phi} & F ((x \circ_j x') \circ_i (y
  \circ_j y')).
}
\end{gather*}
A \emph{monoidal functor} between monoidal pseudo $n$-fold categories
is just a functor between the corresponding pseudo $(n+1)$-fold
categories. Likewise, a \emph{braided $($symmetric$)$ monoidal
  functor} between braided (symmetric, respectively) monoidal pseudo
$n$-fold categories is a functor between the corresponding pseudo
$(n+2)$-fold ($(n+3)$-fold, respectively) categories.

\section{Higher Spans and Cospans}

In this section, we will give basic examples of symmetric monoidal
pseudo $n$-fold categories and prove the following theorem along the
way.

\begin{thm}
  $k$-Spans of sets $($more generally, topological spaces$)$,
  $k$-cospans of sets $($or topological spaces$)$, $k$-cospans of
  $($differential graded$)$ algebras for $k \le n \le
  \infty$ form symmetric monoidal pseudo $n$-fold categories.
\end{thm}

\subsection{Higher Spans of Sets and Topological Spaces}

We will limit the discussion here to higher spans of topological
spaces only. This will cover the case of higher spans of sets, if you
regard sets as spaces with discrete topology. For $k \ge 0$, a
$k$-\emph{span $($correspondence$)$ of topological spaces} is a
commutative diagram of continuous maps between topological spaces,
which is shaped as a barycentric subdivision of a $k$-dimensional
cube:
\[
\xymatrix@=10pt{ & & \bullet \ar@{-}[rr] & & \bullet \ar@{-}[rr] & &
  \bullet \ar@{-}[dd]
  \\
  & \bullet \ar@{-}[ur] \ar@{-}[rr] & & \bullet \ar@{-}[ur]
  \ar@{-}[rr] & & \bullet \ar@{-}[ur] \ar@{-}[dd]
  \\
  \bullet \ar@{-}[ur] \ar@{-}[rr] \ar@{-}[dd] & & \bullet \ar@{-}[ur]
  \ar@{-}[rr] \ar@{-}[dd] & & \bullet \ar@{-}[ur] \ar@{-}[dd] & &
  \bullet \ar@{-}[dd]
  \\
  & & & & & \bullet \ar@{-}[ur] \ar@{-}[dd]
  \\
  \bullet \ar@{-}[rr] \ar@{-}[dd] & & \bullet \ar@{-}[rr] \ar@{-}[dd]
  & & \bullet \ar@{-}[ur] \ar@{-}[dd] & & \bullet
  \\
  & & & & & \bullet \ar@{-}[ur]
  \\
  \bullet \ar@{-}[rr] & & \bullet \ar@{-}[rr]
  & & \bullet \ar@{-}[ur] }.
\]
The edges in the diagram must be maps pointed away from the barycenter
of each $m$-face, $1 \le m \le k$, toward the barycenters of the
bounding $(m-1)$-faces. All the squares of maps in the diagram must
commute. To be included as $k$-morphisms of a pseudo $n$-fold
category, $k$-spans also need to have directions $1 \le i_1, \dots,
i_k \le n$ assigned to the axes of the $k$-cube, so that the two
$(k-1)$-faces bounding the cube in direction $i_p$ are given by the
equations $x_{i_p} = 0$, the source face, and $x_{i_p} = 1$, the
target face. Two $k$-spans may be composed in direction $i_p$,
whenever the target of one $k$-span conicides with the source of the
other in this direction. To compose such $k$-spans, Each sequence
\[
A_1 \leftarrow A_2 \rightarrow A_3 \leftarrow A_4 \rightarrow A_5
\]
of maps in direction $i_p$ that will show up in the diagram for the
composition needs to be converted into a sequence $A_1 \leftarrow A'
\to A_5$ by taking the canonical pullback (a.k.a.\ fibered product)
\[
\xymatrix{
& & A' \ar@{.>}[dl] \ar@{.>}[dr]\\
& A_2 \ar[dl] \ar[dr]^{p_1} & & A_4 \ar_{p_2}[dl] \ar[dr]\\
A_1 & & A_3 & &  A_5,
}
\]
where
\[
A' = A_2 \times_{A_3} A_4 := \{ (a_2, a_4) \in A_2 \times A_4 \; | \;
p_1(a_2) = p_2 (a_4) \}.
\]
Note that each square of maps in the diagram of the resulting span
will still be commutative.

We should be careful to use formal units $\id_i (x)$, rather than
actual $k$-spans with all maps in the $i$th directions being identity
maps. Otherwise, these units would be weak and we will have to burden
our definition of a pseudo $n$-fold category by weak units. Thus,
\emph{formal units} $\id_i (x)$ for all $(k-1)$-morphisms $x$ missing
direction $i$ shall be added to the set of $k$-morphisms. This in fact
means that a $k$-morphism in given directions will be represented by
an $m$-span $x$ for some $m \le k$ and, when $m < k$, thought of as a
formal identity morphism, based on the $m$-span $x$, in the directions
complementary to those present in $x$. Formal units must be composed
as strict units, i.e., in such a way that they satisfy the identities
\[
\id_i s_i (x) \circ_i x = x =  x \circ_i \id_i t_i (x).
\]
One can easily check that the formal units will satisfy the other
axioms \eqref{units2}-\eqref{units5} involving units in the definition
of a pseudo $n$-fold category.

Higher spans described above are objects of the categories of
$k$-morphisms. Morphisms in these categories are (continuous) maps
between higher spans, which are nothing but maps from the respective
vertices of one higher span to another, commuting with all the maps
between the vertices within the spans. The associators and
interchangers, described in the next paragraph, will be morphisms
between higher spans.

The associators and interchangers on compositions of higher spans are
defined at the level of points by rearranging parentheses in the
canonical pullbacks, as follows. Composition of three higher spans in
one direction boils down to composing usual spans like
\[
A_1 \leftarrow A_2 \rightarrow A_3 \leftarrow A_4 \rightarrow A_5
\leftarrow A_6 \to A_7
\]
in that direction. The corresponding associator $\alpha$ will act as
follows:
\[
\alpha ((a_2, a_4),a_6)  := (a_2, (a_4, a_6)), \qquad a_i \in A_i.
\]
If we compose four spans in two directions as per the following diagram:
\[
\begin{CD}
\bullet @<<< \bullet @>>> \bullet @<<< \bullet @>>> \bullet\\
@AAA @AAA @AAA @AAA @AAA\\
\bullet @<<< A_1 @>>> \bullet @<<< A_2 @>>> \bullet\\
@VVV @VVV @VVV @VVV @VVV\\
\bullet @<<< \bullet @>>> \bullet @<<< \bullet @>>> \bullet\\
@AAA @AAA @AAA @AAA @AAA\\
\bullet @<<< B_1 @>>> \bullet @<<< B_2 @>>> \bullet\\
@VVV @VVV @VVV @VVV @VVV\\
\bullet @<<< \bullet @>>> \bullet @<<< \bullet @>>> \bullet,
\end{CD}
\]
the corresponding interchanger will again rearrange parentheses, as
follows:
\[
\beta ((a_1,a_2),(b_1,b_2))  :=  ((a_1,b_1),(a_2,b_2)), \qquad a_i
\in A_i, b_i \in B_i.
\]
The pentagon and both hexagon axioms will hold, because the arrows in
each polygon are defined by the same rearrangement of parentheses for
elements of the sets sitting at vertices of the spans as for the very
spans participating in compositions, e.g., in the pentagon diagram,
the map $\alpha \circ_I \id$ may be defined as follows:
\begin{eqnarray*}
  \alpha \circ_i \id: ((x \circ_i y ) \circ_i z) \circ_i u
  & \to & (x \circ_i (y \circ_iy)) \circ_i u,\\
  (((a_2,a_4),a_6),a_8) & \mapsto & ((a_2,(a_4,a_6)),a_8).
\end{eqnarray*}
Thus, the fact that going from the far left vertex of the polygon to
the far right vertex along the upper and the lower paths results in
one and the same vertex implies that these two paths are equal
set-theoretically, and thereby the diagram commutes.

The associators and interchangers on compositions of morphisms
involving units should be set to equal identity, so that the axioms
\eqref{units-coh} are satisfied.

The argument above ignores the monoidal structure and so far shows
that $k$-spans of topological spaces form a pseudo $n$-fold
category. Essentially, the monoidal structure is given by Cartesian
product, but to add it within our framework, we will need to shift $n$
by 3 and think of $k$-spans of topological spaces in directions $i_1,
\dots, i_k$ between 1 and $n$ as $(k+3)$-morphisms (as well as
degenerate $(l+3)$-morphisms for $l > k$) of the symmetric monoidal
pseudo $n$-fold category in directions $-2, 1, 0, i_1, \dots,
i_k$. The degenerate morphisms play the role of formal units in
positive directions, but we also need to add formal units for the
monoidal structure: a single morphism $\mathbf{1}$ in every
multi-direction missing $0, -1, -2$. Compositions of morphisms in
positive directions are given by composition of higher spans, as
above. Compositions of higher spans in negative directions $-2, -1$,
and 0 are all given by Cartesian products of higher spans and denoted
$\otimes$: one takes Cartesian products of the respective entries in
the composed spans and products of maps for the edges of the product
span. When composing with formal units $\mathbf{1}$ in negative
directions, we should be careful to define composition formally as $x
\otimes \mathbf{1} := x$ and $\mathbf{1} \otimes x := x$, rather than
use the actual one-point space instead of $\mathbf{1}$. This avoids
the issue of weak units for the monoidal structure. For all $i,j \le
0$, we define the interchangers $\beta_{ij}: x \otimes y = (x
\circ_{i} \mathbf{1}) \circ_{j} (\mathbf{1} \circ_{i} y) \to (x
\circ_j \mathbf{1}) \circ_{i} (\mathbf{1} \circ_{j} y) = x \otimes y$
as the identity transformations and the interchangers $\beta_{ij}: x
\otimes y = (\mathbf{1} \circ_{i} x) \circ_j (y \circ_{i} \mathbf{1})
\to (\mathbf{1} \circ_j y) \circ_{i} (x \circ_j \mathbf{1}) = y
\otimes x$ as the canonical set-theoretic interchange law for
Cartesian products. (When either $x$ or $y$ is a unit itself, we will
define the interchanger to be the identity.) The polygonal coherence
axioms will be satisfied for the same reason they were satisfied for
compositions of higher spans above: now we are using Cartesian
product, as opposed to more general fibered product.

\subsection{Higher Cospans of Sets and Topological Spaces}

Here we will ``dualize'' the previous example and consider
$k$-\emph{cospans} (a.k.a. ``\emph{cocorrespondences}'') \emph{of
  topological spaces $($or sets$)$}, which will be commutative
$k$-dimensional cubical diagrams of spaces (respectively, sets) at the
barycenters of all faces and continuous (respectively, arbitrary) maps
directed from to barycenters of faces from the barycenters of the
bounding faces. The monoidal structure will be given by disjoint
union, which may be defined as follows in our set-theoretic universe:
\[
x \coprod y := x \times \{x\} \cup y \times \{y\},
\]
where $\{a\}$ is the one-point set whose only element is $a$. The
monoidal units will again be formal, rather than the empty
set. Compositions of morphisms will be defined by using canonical
set-theoretic pushouts
\[
\xymatrix{
& & A'\\
& A_2  \ar@{.>}[ur]   & & A_4  \ar@{.>}[ul]\\
A_1 \ar[ur] & & A_3 \ar[ul]^{i_1} \ar_{i_2}[ur]& &  A_5\ar[ul],
},
\]
where
\[
A' = A_2 \cup_{A_3} A_4 := (A_2 \coprod A_4) / \{i_1(a_3) \sim i_2
(a_3) \text{ for all } a_3 \in A_3\}.
\]
The rest of the argument goes exactly the same way as for higher spans
above.

\subsection{Higher Cospans of Algebras}

This example may be thought of as another dualization of the example
of higher spans of sets. A $k$-cospan of \emph{differential graded
  $($dg$)$} algebras (over a field) is a commutative $k$-cubical
diagram of dg associative algebras placed at the barycenters of all
faces and dg-algebra morphisms directed toward the barycenter of each
face from the barycenters of the bounding faces. The monoidal
structure is given by tensor product $x \otimes y$ of the dg algebras
sitting at the respective vertices of the two $k$-cospans. The
monoidal units will again be formal, rather than the ground field
literally. Compositions of morphisms will be defined by using
canonical tensor products over a dg associative algebra: $A_2
\otimes_{A_3} A_4$. The rest of the argument repeats the arguments for
higher spans and cospans above.

\subsection{Higher Spans of Coalgebras}

This example is a linearization of the example of higher spans of sets
or topological spaces and also linear dual to the example of higher
cospans of algebras. More precisely, if you have a higher span of
toplogical spaces, applying the functors of chains or homology, you
get higher spans of dg coalgebras. A $k$-span of dg coalgebras over a
field is a commutative $k$-cubical diagram of dg coassociative
coalgebras sitting at the barycenters of all faces and dg-coalgebra
morphisms directed from the barycenter of each face to the barycenters
of the bounding faces. The monoidal structure is given by tensor
product of the dg coalgebras. The monoidal units are formal, and
composition of spans is defined using the tensor product of dg
coalgebras relative to a morphism to another coalgebra.

\section{Cobordisms with Corners}
\label{cobordisms}

In this section we are going to prove the following theorem.
\begin{thm}
\label{cob-thm}
Abstract $k$-cobordisms for $k \le n \le \infty$ form a symmetric
monoidal pseudo $n$-fold category.
\end{thm}

\begin{rem}
  There is a notion of cobordisms with corners, embedded nicely in
  coordinate rectangular solids in $\R^\infty$, which leads to a
  \emph{strict} $n$-fold category, see Tillmann \cite{tillmann}. We
  will not use it here, as we are ultimately interested in extended
  TQFTs, which will be functors from a suitable $n$-fold category of
  cobordisms with corners to one of the pseudo $n$-fold categories of
  higher (co)spans. We are not aware of any strict model of the
  latter, therefore there will be little benefit in using a strict
  model of the former.
\end{rem}

An $n$-dimensional manifold with corners is roughly a space smoothly
modeled on $\R_+^n$, where $\R_+ = [0, \infty)$ is the closed positive
real line.

\begin{df}
  A (smooth) $n$-\emph{manifold with corners} is a second-countable,
  Hausdorff topological space $M$ with a sheaf $\F$ of $\R$-algebras,
  which is a subsheaf of the sheaf of $\R$-algebras of continuous
  real-valued functions on $M$, such that $(M,\F)$ is locally
  isomorphic to $(\R_+^n, C^\infty)$. For each $k = 0, 1, \dots, n$,
  an \emph{open $k$-face} of $M$ is a connected component of the
  locally closed submanifold of $M$ consisting of points corresponding
  to an open $k$-dimensional face of $\R_+^n$. A \emph{$($closed$)$
    $k$-face} of $M$ is the closure of an open $k$-face.
\end{df}

\begin{df}
  An $n$-\emph{cobordism $M$ with corners}, see Figure~\ref{whistle},
  is a compact $n$-manifold with corners with the following extra
  data. The cobordism $M$ must be locally modeled on $(I^n,
  C^\infty)$, so that the information about which open face of the
  cube each point of $M$ corresponds to is part of the data. Thus,
  each open (and thereby closed) $k$-face of $M$ is labeled by a
  combination $\{i_1, i_2, \dots, i_{n-k}\}$ of $n-k$ out of $n$
  numbers $1, 2, \dots, n$ and a sequence $(\eps_{i_1}, \dots,
  \eps_{i_{n-k}})$ of zeros and ones, which define the $k$-face of the
  cube $I^n$ by the equations
  \begin{equation}
    \label{k-face}
    x_{i_1} = \eps_{i_1}, x_{i_2} = \eps_{i_2}, \dots, x_{i_{n-k}}
    = \eps_{i_{n-k}}.
  \end{equation}
  These labelings must be compatible between faces in the sense that
  each $k$-face labeled by the equations \eqref{k-face} must locally
  be the intersection of $k$ $(n-1)$-faces labeled by the equations
  $x_{i_1} = \eps_{i_1}, x_{i_2} = \eps_{i_2}, \dots, x_{i_{n-k}} =
  \eps_{i_{n-k}}$. Each $k$-face $F$ of $M$ must be provided with a
  \emph{collar}, which is the germ of a diffeomorphism of an open
  neighborhood of $F \times \{(\eps_{i_1}, \dots, \eps_{i_{n-k}})\}$
  in $F \times I^{n-k}$ with an open neighborhood of $F$ in
  $M$. Coordinates used in this copy of $I^{n-k}$ are $(x_{i_1},
  x_{i_2}, \dots, x_{i_{n-k}})$. Here a \emph{germ} is an equivalence
  class of such diffeomorphisms; two such diffeomorphisms of open
  neighborhoods of $F$ in $M$ are \emph{equivalent}, if they are equal
  over the intersection of the neighborhoods. The collars at different
  faces must be compatible with each other, that is, for any
  $(k-1)$-subface $F'$ of each $k$-face $F$ given by the equation
  $x_{i_{n-k+1}} = \eps_{i_{n-k+1}}$, the restriction of the collar $F
  \times I^{n-k} \to M$ to $F' \times I^{n-k}$ must coincide with the
  restriction of the collar $F' \times I^{n-k+1} \to M$ to the subset
  $F' \times I^{n-k}$ given by the equation $x_{i_{n-k+1}} =
  \eps_{i_{n-k+1}}$.

  For each $j =1, \dots, n$, the $(j,+)$-\emph{boundary} $\del^+_j M$
  of $M$ is the union of all the $(n-1)$-faces given by the equation
  $x_j = 1$. Similarly, the $(j,-)$-\emph{boundary} $\del^-_j M$ is
  the union of the $(n-1)$-faces given by the equation $x_j=0$. The
  $(j,+)$-boundary inherits the structure of an $(n-1)$-cobordism with
  corners, with the model $(n-1)$-cube $I^{n-1} \subset I^n$ given by
  the equation $x_j = 1$. The same for the $(j,-)$-boundary and the
  equation $x_j = 0$.
\end{df}

\begin{figure}
\includegraphics[width=7cm]{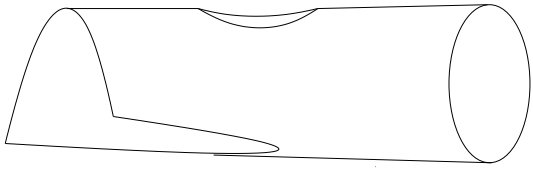}
\caption{The Whistle}
\label{whistle}
\end{figure}

\begin{rem}
  An $n$-cobordism with corners is automatically a \emph{manifold with
    faces} under the terminology of \cite{jaenich,laures} and moreover
  a \emph{manifold with $n$ distinguished faces} $\del^+_1 M \cup
  \del^-_1 M, \dots, \del^+_n M \cup \del^-_n M$, see
  \cite{grandis-III}, also known as an $\langle n
  \rangle$-\emph{manifold}, see \cite{jaenich,laures}.
\end{rem}

\begin{sloppypar}
Abstract $n$-cobordisms $M$ and $N$ with corners may be
\emph{composed} to an $n$-cobordism $M \circ_j N$ with corners along
the $j$th boundary for any $j = 1, \dots, n$, if the $(j,+)$-boundary
$\del^+_j M$ coincides with the $(j,-)$-boundary $\del^-_j N$ of $N$
(literally, as sets with the structure of a smooth $(n-1)$-cobordism
with corners). As a topological space, $M \circ_j N$ is the canonical
topological pushout
\[
\begin{CD}
  F_j @>{f_+}>> N\\
  @V{f_-}VV @VVV\\
  M @>>> M \circ_j N,
\end{CD}
\]
where $F_j = \del^+_j M = \del^-_j N$ is the \emph{seam}, the common
boundary of $M$ and $N$ along which the composition occurs. The
pushout is defined as a quotient of the disjoint union, as follows:
\[
M \circ_j N := M \coprod N /\sim,
\]
where $x \sim y$, if there exists $z \in F_j$ such that $f_- (z) = x$
and $f_+ (z) = y$. The smooth structure along the seam $F_j$ is
defined by requiring that the map of a neighborhood of $F_j \times
\{0\}$ in $F_j \times \R$ to $M \circ_j N$ given by taking the union
of the collar of $F_j$ in $M$ (after a coordinate change $x_j \mapsto
x_j - 1$) and the collar of $F_j$ in $N$ be smooth. The faces of the
composition cobordism $M \circ_j N$ are just the unions of faces of
$M$ and $N$ except the faces belonging to the seam, which are regarded
as canceling out. The faces of $M \circ_j N$ are labeled the same way
as their components on $M$ and $N$. These labels at the
lower-dimensional faces on the seam $F_j$ also match, because it was
assumed to be equal to the outgoing $j$th boundary of $M$ and the
incoming $j$th boundary of $N$ as a cobordism with corners, including
labelings of all the faces. The collars on the faces are obtained from
the collars of the corresponding faces of $M$ and $N$. The
$(i,+)$-boundary of $M \circ_j N$ is the union of the
$(i,+)$-boundaries of $M$ and $N$ for $i \ne j$ and just the
$(i,+)$-boundary of $N$ for $j=i$. Analogously, the $(i,-)$-boundary
of $M \circ_j N$ is the union of the $(i,-)$-boundaries of $M$ and $N$
for $i \ne j$ and just the $(i,-)$-boundary of $M$ for $j=i$.
\end{sloppypar}

Thus, a $k$-cobordism with corners for $k \le n$ (whose directions are
labeled by the choice of a coordinate $k$-cube $I^k \subset I^n$) will
be regarded as a $k$-morphism of the pseudo $n$-fold category of
cobordisms. More precisely, a $k$-cobordism with corners will be an
object of the category of $k$-morphisms, with morphisms given by
smooth maps, respecting the corner structure, i.e., mapping faces to
faces, preserving labelings and collars. We will also need to add
formal units with respect to composition, which will result in
treating $m$-cobordisms with corners for $m < k$ as identity
$k$-morphisms in specified complementary directions, as in the example
of cospans of topological spaces above.

The monoidal structure will be given by disjoint union of cobordisms
with corners.

\begin{proof}[Proof of Theorem $\ref{cob-thm}$]
  Abstract cobordisms with corners are nothing but higher cospans of
  sets with certain structure. Given that higher cospans of sets form
  a symmetric monoidal pseudo $n$-fold category, we just need to make
  sure that the structure glues under composition, i.e., the smooth
  structure is naturally defined on the composition of cobordisms with
  corners, as explained in the construction of cobordism composition
  above. This fact is guaranteed by the collars. Smooth maps between
  cobordisms with corners glue under compostion as well.
\end{proof}

\section{Comparison to Quasi-Categories}
\label{lurie}

Cobordisms with corners can also be handled by quasi-categories
\cite{joyal:quasi2002} (also known as $\infty$-categories
\cite{lurie:topoi} or weak Kan complexes
\cite{boardman-vogt:structures}). In our approach, the objects were
compact 0-dimensional manifolds, i.e., finite sets, the 1-morphisms
were represented by compact one-dimensional manifolds with boundary,
i.e., disjoint unions of intervals and circles, the 2-morphisms were
given by surfaces with corners, \dots, the $n$-morphisms were given by
$n$-dimensional cobordisms with corners. Remember that $k$-morphisms
in a pseudo $n$-fold category actually make up a usual category; what
we have just described are the objects in the categories of
$k$-morphisms, the morphisms are given by smooth maps, respecting the
cobordism structure (the corners, labels, and collars). A
quasi-category version of the category of $n$-cobordisms with corners
will have $(n-1)$-dimensional manifolds with corners as objects,
$n$-dimensional cobordisms with corners as 1-morphisms, pairs of
composable $n$-cobordisms with corners along with a choice for their
composition (a pushout) as 2-morphisms, and certain towers of pushouts
(all being $n$-dimensional manifolds with corners) of $k$ composable
cobordisms as $k$-morphisms for higher values of $k$. While this
inclusion of $n$-cobordisms with corners into the framework of
quasi-categories might be interesting on its own, it largely ignores
cobordisms of dimension other than $n-1$ and $n$ for a fixed number
$n$. To include other dimensions along these lines, one needs a
mixture of the $n$-fold (or $\infty$-fold) category approach with that
of quasi-categories, i.e., some kind of an $n$-fold quasi-category,
which is what we are going to introduce in this section.

We would like to outline the following simplicial approach to weak
$n$-fold categories, in the spirit of Boardman-Vogt's approach to weak
$n$-categories, see \cite{boardman-vogt:structures, joyal:quasi2002,
  lurie:topoi}. We would call them $n$-fold quasi-categories, where $0
\le n \le \infty$.  An \emph{$n$-fold quasi-category} is a $n$-fold
simplicial set $\mathbf{S}: (\mathbf{\Delta}^{\op})^n \to \Set$
satisfying the \emph{weak Kan} or \emph{inner horn-filling condition}:
for each sequence $\sigma_1, \dots, \sigma_n$, where each $\sigma_i$
is either a simplex $\Delta^{k_i}$, $k_i \ge 0$, or an inner horn
$\Lambda_j^{k_i} \subset \Delta^{k_i}$, $0 < j < k_i$, any inner
multihorn $(\sigma_1, \dots, \sigma_n) \to \mathbf{S}$ may be
completed to a multisimplex $(\Delta^{k_1}, \dots, \Delta^{k_n}) \to
\mathbf{S}$. (Compare this to the unique inner horn-filling condition
in characterizing the nerve of a strict $n$-fold category in
Theorem~\ref{nerve}.) Equivalently, one can iteratively define an
$n$-fold quasi-category as a weak Kan (or quasi-category) object in
the category of $(n-1)$-fold quasi-categories, starting with $0$-fold
quasi-categories, which are just sets by definition. Simply put, an
$n$-fold quasi-category is what it is, that is to say, an $n$-fold
iteration of the notion of a quasi-category. The notion of an $n$-fold
quasi-category might be regarded as a model of an
$(\infty,n)$-category, alternative to Barwick and Lurie's
\cite{lurie:tft} version of it as an $n$-fold (complete) Segal space.

An $n$-fold quasi-category $\Cob$ of $n$-cobordisms with corners may
be described by defining a multisimplex $(\Delta^{k_1}, \dots,
\Delta^{k_n}) \to \Cob$ as a multisimplicial diagram of composable
$n$-dimensional cobordisms with corners. Start withs a pasting diagram
in the form of an $n$-dimensional rectangular solid consisting of
strings of $k_1$ composable $n$-cobordisms in direction $1$, $k_2$
composable $n$-cobordisms in direction $2$, \dots, $k_n$ composable
$n$-cobordisms in direction $n$. Then add choices of all possible
pushouts, so that for each triangular prism $[p_1 q_1] \times \dots
\times [p_{i-1} q_{i-1}] \times [p_i q_i r_i] \times [p_{i+1} q_{i+1}]
\times \dots \times [p_n q_n]$, $p_i < q_i < r_i$, in the
multisimplex, the cobordism along the face $[p_1 q_1] \times \dots
\times [p_i r_i] \times \dots \times [p_n q_n]$ is a pushout of the
cobordisms along the faces $[p_1 q_1] \times \dots \times [p_i
q_i]\times \dots \times [p_n q_n]$ and $[p_1 q_1] \times \dots \times
[q_i r_i] \times \dots \times [p_n q_n]$ (not necessarily the
canonical way we used in Section~\ref{cobordisms}). Inner multihorns
will be similar towers of pushouts, possibly missing the pushouts
along an inner face of the $k_i$-simplex in the $i$th direction for
some values of $i$, $1 \le i \le n$. Since pushouts for $n$-cobordisms
with corners exist (referred to as compositions of cobordisms with
corners in Section~\ref{cobordisms}), all multihorns in triangular
prisms can always be filled. As concerns more general multihorns, we
can use canonical pushouts to add missing edges and coherence maps
$\alpha$ and $\beta$ to realize missing faces as pushouts. The
existence of such compositions is guaranteed by our Weak Coherence
theorem, Theorem~\ref{existence}.

\begin{rem}
  In the above we have assumed implicitly that all $k_1, \dots, k_n
  \ge 1$. If one of the $k_i$'s is equal to 0, we are talking about an
  $(n-1)$-cobordism with corners in all the directions except $i$. If
  more of the $k_i$'s are 0, it will be a cobordism of lower
  dimension, missing these directions.
\end{rem}
The above arguments prove the following result.

\begin{thm}
  $n$-Cobordisms with corners form an $n$-fold quasi-category.
\end{thm}

Of course, one would also be interested in taking into account the
monoidal structure coming from disjoint union. This might be addressed
via $(n+1)$-fold (or higher) quasi-categories in a way similar to what
we did with monoidal pseudo $n$-fold categories above, but we would
rather not discuss this issue, as we will not be using
quasi-categories and Segal spaces in this paper. In the case of Segal
spaces, Lurie discusses the notion of symmetric monoidal
$(\infty,1)$-categories in \cite{lurie:dag3}. Another feature that we
would like to have in the context of $n$-fold quasi-categories is
including smooth maps and in particular diffeomorphisms in the
picture. It is not obvious how to do that, whereas smooth maps are
organic part of the treatment of cobordisms within the framework of
pseudo $n$-fold categories, developed in this paper.

\section{Topological Quantum Field Theories}
\label{TQFT}

\begin{metathm}
\label{meta}
  An \emph{$($extended$)$ $n$-dimensional Topological Quantum Field
    Theory $($TQFT$)$} is a symmetric monoidal functor from the
    symmetric monoidal pseudo $n$-fold category of cobordisms with
    corners to the opposite of the symmetric monoidal pseudo $n$-fold
    category of higher spans of sets (topological spaces, manifolds,
    orbifolds, or stacks) or to the symmetric monoidal pseudo $n$-fold
    category of higher cospans of dg associative algebras.
\end{metathm}

A functor to the opposite of a category may be called
\emph{contravariant}, as usual. Here we will be talking about TQFTs as
contravariant functors, having in mind that by taking appropriate
function algebras or de Rham or (singular) cochain algebras in
concrete situations, the contravariant functor with values in higher
spans of sets (or more interesting geometric objects) will induce a
(covariant) functor with values in higher cospans of (dg) algebras.

The rest of the section is a case study, meant to be a ``metaproof''
of the metatheorem. We will discuss gauge theory in greater detail and
then briefly go over Wess-Zumino-Witten (WZW) theory and a nonlinear
sigma model, including perhaps the most general AKSZ model, all of
which will repeat the same features. In principle, all the three basic
theories may be regarded as variations on the theme of sigma models: a
sigma model is based on maps $M \to X$ from cobordisms $M$ with
corners to a fixed target space $X$, whereas the WZW model is based on
maps $M \to G$ to a Lie group and gauge theory on maps $M \to BG$ to
the classifying space $BG$ of a Lie group. The fact that they all give
rise to TQFT functors in the sense of Metatheorem~\ref{meta} is,
roughly speaking, based on the fact that maps $M \to X$ restrict
naturally to faces of $M$, as well as to the components of $M$, if it
happens to be a composition of cobordisms or disjoint union. Our
formalism also allows us to consider mixed theories, for example,
gauge theory combined with a sigma model, based on maps $M \to X
\times BG$, or twisted ones, such as a theory based on pairs of maps
$M \to X \to BG$.

We would also like to make a few remarks about quantum versus
classical "path space" $\Phi(M)$. In our description, we will
construct the "quantum path space" $\Phi(M)$ first and then describe
examples of action functionals defined on the quantum path space. The
Euler-Lagrange equations for these functionals define "classical
trajectories," which form a "classical path space" $\Phi_{\cl} (M)
\subset \Phi (M)$. We name the TQFTs obtained from $\Phi_{\cl} (M)$
after the action functionals, e.g., the Yang-Mills functional gives
rise to Yang-Mills theory. In principle, to obtain TQFTs in a more
classical, operator formalism, one wishes to perform Feynman
integration over the whole path space $\Phi (M)$, using the action
functional as a density function. This integral localizes around the
classical path space $\Phi_{\cl} (M)$ via Feynman expansion. Usually,
a combination of fiberwise integration (e.g.., Faddeev-Popov or BV
quantization procedures) and further localization is performed to
reduce the path integral to a finite-dimensional one or and finally
obtain a single vector or operator between state spaces. However, this
procedure depends on the specifics of the setup: sometimes getting
such a vector is problematic because of divergences or anomalies,
sometimes one gets finite-dimensional spaces of such vectors instead
of uniquely defined vectors (cf.\ conformal blocks and modular
functors). The goal of our approach is to bring to surface a common
geometric background featured by all TQFTs irrespective of their
particulars.

\subsection{Gauge Theory}

Fix a compact Lie group $G$ and, for any $k$-cobordism $M$ with
corners, consider the set
\[
\Phi(M) = \{ (P, \nabla) \}/ \sim,
\]
where $P$ is a (smooth) principal $G$-bundle over $M$, $\nabla$ is a
connection on $P$, and $\sim$ denotes \emph{gauge equivalence}, given
by $G$-equivariant isomorphisms $P \stackrel{\sim}{\to} P'$ lifting
the identity isomorphism from the base $M$ and respecting the
connections on $P$ and $P'$.

The correspondence $M \mapsto \Phi(M)$ may be completed to a functor, a
TQFT, to be more precise. First of all, we need to introduce more
structure on $\Phi(M)$. Recall that connections on $P$ form an affine
space (i.e., a principal homogeneous space) over the vector space
$C^\infty (M, \Omega^1 (\gtg_P))$ of $\gtg_P = (P \times_{\Ad G}
\gtg)$-valued 1-forms, $\gtg$ being the Lie algebra of $G$ and $\Ad G$
the adjoint action of $G$ on $\gtg$. Using an inner product on $\gtg$,
we can introduce topology on the vector space and the affine space
over it and thereby the set
\[
\Phi_P(M) = \{ \text{connections  $\nabla$ on $P$}  \}/ G^M
\]
of connections on a fixed principal $G$-bundle $P$, up to the action
of the gauge group $G^M = C^\infty (M,G) = H^0(M, \GG)$, where $\GG :=
\CC^\infty (M,G)$ is the sheaf of (germs of) smooth functions $M \to
G$, and $C^\infty (M,G)$ is the space of global smooth functions $M
\to G$. The gauge group $G^M$ might also be understood as
\emph{non-abelian cohomology} $H^0(M, \GG)$ in degree zero. Finally,
redefine $\Phi(M)$ as an affine bundle over the space of gauge classes
of principal $G$-bundles $P$ with fiber $\Phi_P(M)$, i.e.,
\[
\Phi(M) = \coprod_{P \in H^1(M, \GG)} \Phi_P (M) = \coprod_{P \in [M, BG]}
\Phi_P (M),
\]
where $P$ in the first disjoint union runs over the set $H^1(M, \GG)$,
\emph{non-abelian cohomology} in degree one, which is, by definition,
the set of gauge classes of principal $G$-bundles over $M$, and
$[M,BG]$ in the second disjoint union is the set of homotopy classes
of maps from $M$ to the classifying space $BG$ of the Lie group
$G$. Moreover, instead of treating $\Phi(M)$ as a topological space,
one has to adopt the ``stacky'' point of view and regard $\Phi(M)$ as
a space with a group action, rather than an orbit space. In this case,
it will be an affine bundle over $H^1(M, \GG)$ with fibers being
affine spaces over the infinite dimensional vector space $C^\infty (M,
\Omega^1 (\gtg_M))$, endowed with, say, Fr\'echet topology, with an
action of a (generally, infinite dimensional) Lie group, $G^M$. This
suggests that one can view
$\Phi(M)$ as a step toward the Borel quotient
\[
H^1(M, \GG) / H^0(M, \GG).
\]

To extend the correspondence $M \mapsto \Phi(M)$ to a TQFT, we would
like to associate an $n$-span of stacks to a given $n$-cobordism $M$
with corners. The core (the object at the center) of the $n$-span will
be the stack $\Phi(M)$ of gauge classes of connections on principal
$G$-bundles. The object at the barycenter of each face of the $n$-span
will be the stack of gauge classes of connections on principal
$G$=bundles over the union of faces of the $n$-cobordism labeled by
the face of the $n$-cube $I^n$ corresponding to the face of the
$n$-span. The maps from the barycenter of the face toward the
barycenters of the bounding faces will be the restriction map: a
connection on a principal $G$-bundle over a cobordism with corners
induces a principal $G$-bundle with a connection on all the source and
target manifolds of the cobordism. This restriction map is manifestly
\emph{equivariant} with respect to the notion of gauge equivalence. We
will abuse the notation and let $\Phi(M)$ denote the resulting
$n$-span, rather than just its core. The correspondence $M \mapsto
\Phi(M)$ is in fact a functor, as it takes $k$-cobordisms with corners
to $k$-spans of sets, functorially with respect to smooth maps of
cobordisms and maps of spans, whereas compositions $M \circ_j N$ of
cobordisms induce natural coherence maps
\[
\Phi(M \circ_j N)  \to \Phi(M) \circ_j \Phi(N),
\]
where on the left-hand side, we have the set of (classes of)
connections over the composition $M \circ_j N$, whereas on the
right-hand side, we have the set of pairs of (classes of) connections
on $M$ and $N$, coinciding (up to gauge equivalence) on the common,
$j$th boundary of $M$ and $N$. The flip of the direction of the map as
compared to Section~\ref{functors} is due to the fact that $\Phi$ is
contravariant. It is easy to see that the coherence conditions of
Section~\ref{functors} are satisfied.

We could have imposed boundary conditions on principal $G$-bundles and
connections along the collars, so that they glue nicely when the
cobordisms are composed. The natural choice of \emph{boundary
  conditions} would be compatibility with collars at all faces of the
cobordism $M$: the pair $(P, \nabla)$ along a collar must be a trivial
extension of a principal $G$-bundle with connection from the face to
its collar. Gauge transformations would have to be constant along the
collars. Gluing along composable cobordisms would result in coherence
maps
\[
\Phi(M) \circ_j \Phi(N) \to \Phi(M \circ_j N),
\]
which would mean that the TQFT functor is a \emph{cofunctor}, rather
than a functor, by definition. However, we anticipate that since we
are talking about gauge classes of connections, irrespective of the
boundary conditions, the TQFT functor is actually a \emph{strong
  functor}, that is to say, all coherence maps are invertible.

We also need to see that the gauge theory functor $M \mapsto \Phi(M)$
is symmetric monoidal. Remember that now we have to think about
$k$-cobordisms with corners and $k$-spans of sets as $(k+3)$-morphisms
spanning $k$ positive, spatial, directions and the three nonpositive
ones: $-2, -1$, and $0$. Compositions in the nonpositive directions
are all given by disjoint union for cobordisms and Cartesian product
for spans. Thus, for $j \le 0$ the coherence maps
\[
\Phi(M \circ_j N)  \to \Phi(M) \circ_j \Phi(N),
\]
will represent (the gauge class of) a principle $G$-bundle with a
connection on $M \coprod N$ as a pair $((P_M, \nabla_M), (P_N,
\nabla_N))$ of the same on $M$ and $N$ separately. Again, the
coherence conditions of Section~\ref{functors} involving nonpositive
directions is easy to check.

There exist different reincarnations of gauge theory, whose
description within our formalism is given below.

\subsubsection{The Dijkgraaf-Witten Toy Model}

This is a version of gauge theory in the case when the group $G$ is
finite. We will indicate the main ingredients of the construction,
leaving out those details which are similar to those in the case of
the general treatment of gauge theory above.

The functor $M \mapsto \Phi(M)$ assigns to a cobordism $M$ with
corners the set
\[
\Phi(M) = \Hom (\Pi_1 (M), \G)/ \sim
\]
of morphisms from the fundamental groupoid $\Pi_1(M)$ of $M$ to the
groupoid $\G$ of principal homogeneous spaces for group $G$, up to
natural isomorphism. The objects of the fundamental groupoid
$\Pi_1(M)$ are points of $M$, and the morphisms are homotopy classes
of paths in $M$. The objects of $\G$ are \emph{principal
  homogeneous spaces$,$ or torsors}, i.e., sets with free, transitive
(left) action of $G$. The morphisms of $\G$ are $G$-equivariant
maps. Again, we should think about $\Phi(M)$ as a stack, rather than a
set: it is the groupoid of functors $\Pi_1(M) \to \G$ with natural
isomorphisms as morphisms. One can represent this stack as a global
quotient by the "conjugation" action of the group $G$ via
automorphisms of the groupoid $\G$: for an element $g \in G$ and a
$G$-torsor $X$, a new torsor $gX$ is defined as the set $\{ gx \; | \;
x \in X\}$ along with $G$-action, as follows:
\[
h (gx) := (hg)x = g(g^{-1}hg) x \qquad \text{for $h \in G$ and $x \in X$.}
\]
The isotropy subgroup of a point will then consist of those elements
of $G$ which commute with the monodromy group.

As in the general case of gauge theories above, the $G$-space
$\Phi(M)$ should be extended to a higher span of $G$-spaces, with the
span structure induced by restriction of morphisms $\Pi_1 (M) \to \G$
to faces of $M$.

In this case, it is easy to see that the coherence maps
\[
\Phi(M \circ_j N)  \to \Phi(M) \circ_j \Phi(N),
\]
will be $G$-equivariant isomorphisms: a pair of morphisms $\Pi_1 (M) \to
\G$, $\Pi_1 (N) \to \G$, equal on the common boundary of $M$
and $N$, extend uniquely to a morphism $\Pi_1 (M \circ_j N) \to
\G$ by Seifert-van~Kampen argument.

In a similar way, the coherence maps for $\circ_j$ compositions for $j \le 0$, i.e., the natural maps
\[
\Phi(M \coprod N)  \to \Phi(M) \times \Phi(N)
\]
are also isomorphisms of spans.

\begin{rem}
  A key ingredient of Dijkgraaf-Witten's toy model is the choice of a
  class $\alpha$ in the cohomology $H^3 (BG, U(1))$ of the classifying
  space $BG$ under the assumption that $\dim M = 3$. This class may be
  used to obtain a "partition function," also known as the
  Dijkgraaf-Witten invariant
  \[
  Z(M) := \sum_{\gamma \in \Phi(M)} \gamma^* (\alpha) ([M]) \in U(1),
  \]
  where we use the additive notation for $U(1) = \R/ \Z$ and
  understand summation in the orbifold sense, i.e.,
\[
\sum_{\gamma \in \Phi(M)} \gamma^* (\alpha) ([M]) := \frac{1}{\abs{G}}
\sum_{\gamma \in \Hom (\pi_1(M), G)} \gamma^* (\alpha) ([M]),
\]
representing $\Phi(M)$ as a global quotient $\Hom (\pi_1(M), G) /G$ by
the conjugation action of $G$. More generally, if $\del_2 M = \del_3 M
= \varnothing$, then one can define a TQFT in a more traditional
sense. This includes metrized complex lines $L(\del_1^- M)$ and
$L(\del_1^+ M)$ corresponding to the incoming and outgoing boundary,
respectively, and a "tunneling amplitude," a linear map
\[
\Phi_M: L(\del_1^- M) \to L(\del_1^+ M),
\]
see \cite{dijkgraaf-witten,freed-quinn,fhlt}. The above summation
formula is a simple case of path integration, which is needed to
obtain a number, such as $Z(M)$, or a map, such as $\Phi_M$, from the
span $\Phi(M)$.
\end{rem}

\subsubsection{Yang-Mills Theory}

In Yang-Mills theory, one starts with the \emph{Yang-Mills action
  $($functional$)$}:
\[
\int_M \Tr (F \wedge *F),
\]
defined on the space $\{ (P, \nabla) \}/ \sim$ of equivalence classes
of principal $G$-bundles $P$ with connection $\nabla$ for a compact
Lie group $G$. To define this functional, we need to consider
cobordisms $M$ with corners and a Riemannian metric which makes all
faces orthogonal to each other at intersections, so that the Hodge
dual $*F$ of the curvature form $F$ of the connection $\nabla$ is
defined, whereas $\Tr$ denotes the Killing form on the Lie algebra
$\gtg$. The Euler-Lagrange equations
\[
d_\nabla F = d_\nabla *F = 0,
\]
for the Yang-Mills functional are manifestly gauge invariant and
called the \emph{Yang-Mills equations}. Then the functor
$\Phi_{\YM}(M)$ defining Yang-Mills theory as a TQFT consists of all
solutions to the Euler-Lagrange equations, i.e.,
\[
\Phi_{\YM}(M) = \{ (P, \nabla) \; | \; d_\nabla *F = 0 \}/ \sim.
\]
The first Yang-Mills equation $d_\nabla F = 0$ is the Bianchi
identity, which is satisfied automatically, and thereby redundant. As
above, it is easy to check that we get a TQFT in the sense of this
paper.

\subsubsection{Yang-Mills Theory in Lower Dimensions}

\emph{Donaldson theory} is a version of Yang-Mills theory in the case
of $\dim M =4$. The advantage of four dimensions is that in this case
both $F$ and $*F$ are 2-forms and moreover the Yang-Mills functional
minimizes on \emph{self-dual} and \emph{anti-self-dual connections},
also known as \emph{instantons} --- those whose curvature $F$
satisfies
\[
*F = \pm F.
\]
Thereby this equation implies the Yang-Mills equations. In particular,
Donaldson considered anti-self-dual connections. Thus, one can define
Donaldson theory as a 4d TQFT (i.e., one defined on the pseudo
four-fold category of cobordisms of dimension $\le 4$) by specifying
\[
\Phi_{\D}(M) = \{ (P, \nabla) \; | \;  *F = -F \}/ \sim
\]
for $\dim M = 4$ and using the standard Yang-Mills equation $d_\nabla
* F = 0$ for $\dim M \le 3$.  When $\dim M = 3$, Yang-Mills theory is
known to turn into \emph{Chern-Simons theory} with the corresponding
TQFT functor (on the pseudo three-fold category of cobordisms of
dimension $\le 3$) defined as follows:
\[
\Phi_{\CS}(M) = \{ (P, \nabla) \; | \;  F = 0 \}/ \sim.
\]
\emph{Seiberg-Witten theory} is a version of Yang-Mills theory in four
dimensions, alternative to Donaldson theory. Seiberg-Witten theory is
described by certain gauge fields, called \emph{monopoles}, on
4-manifolds with complex spin structure. Restriction of monopoles to
the boundary also generates a 4d TQFT in the sense of the present
work.

\subsection{Wess-Zumino-Witten Theory}

\begin{sloppypar}
In \emph{Wess-Zumino-Witten $($WZW$)$ theory} one starts with a
compact Lie group $G$ and assigns to an $n$-cobordism $M$ the mapping
space
\[
\Phi(M) := \{ \text{smooth maps } M \to G \}/\sim,
\]
where two maps $f$ and $g$ are equivalent, if they fit into a
commutative diagram
\[
\xymatrix{
M \ar[d]_\phi \ar[r]^f & G,\\
M \ar[ru]_g
}
\]
where $\phi$ is a diffeomorphism respecting the structure of a
cobordism with corners. Out of a given $n$-cobordism, we are actually
getting an $n$-span $\Phi(M)$ of stacks, where the maps making up the
span come from restriction of maps to $G$ from faces of $M$ to faces
of lower dimension. The coherence maps
\begin{eqnarray}
\label{coherence1}
\Phi(M \circ_j N) & \to & \Phi(M) \circ_j \Phi(N),\\
\label{coherence2}
\Phi(M \coprod N)  & \to & \Phi(M) \times \Phi(N)
\end{eqnarray}
are given by restriction from the composition (or union) of two
cobordisms to the individual cobordisms.
\end{sloppypar}

WZW theory for a complex semi-simple Lie group $G$ and 2-cobordisms
$M$ provided with complex structure is defined by an action whose
Euler-Lagrange equations require that solution fields $f: M \to G$ be
holomorphic. Thus, under these assumptions, we can consider only
classical trajectories in our path space and redefine the TQFT functor
on the pseudo 2-fold category of 2-cobordisms with complex structure
as follows:
\[
\Phi_{\cl} (M) := \{ \text{holomorphic maps } M \to G \}/\sim,
\]
where two maps are equivalent, if they are related by a complex
isomorphism of $M$ preserving the cobordism with corners structure.

\subsection{Sigma Model}

In a \emph{$($nonlinear$)$ sigma model}, one fixes a \emph{target
  space} $X$, usually a smooth compact manifold, and defines a TQFT,
using smooth maps to $X$:
\[
\Phi(M) := \{ \text{smooth maps } M \to X \}/\sim,
\]
with the same equivalence relation as in the WZW model above: two maps
$f$ and $g$ are equivalent, if they fit into a commutative diagram
\[
\xymatrix{
M \ar[d]_\phi \ar[r]^f & X,\\
M \ar[ru]_g
}
\]
for a diffeomorphism $\phi$. As above, restriction to faces defines a
span $\Phi(M)$ of stacks and restriction to components defines
coherence maps \eqref{coherence1}--\eqref{coherence2}.

If we consider 2d cobordisms with complex structure and a complex
target space $X$, we will be getting what is known as
\emph{$($open-closed$)$ Gromov-Witten theory}, defined by a
``classical path space''
\[
\Phi_{\GW} (M) := \{ \text{holomorphic maps } M \to X \}/\sim,
\]
the equivalence relation given by complex isomorphisms preserving the
structure of cobordism with corners.

Another useful theory, called the \emph{Rozansky-Witten model}
\cite{rozansky-witten} comes in 3 dimensions ($\dim M = 3$). In the
Rozansky-Witten model, one considers oriented 3d cobordims $M$ with
corners, a complex symplectic manifold $X$ as target space, maps $f: M
\to X$, and sections $\eta: M \to f^* T^{0,1} X$ and $\rho: M \to f^*
T^{1,0} X \otimes T^* M$, see more details in \cite{kapustin}. To
accommodate this theory, one can modify the TQFT functor to include
the whole quantum path space $\Phi (M) := \{ (f, \eta, \rho) \}/\sim$
or restrict the attention to classical paths: $\Phi_{\RW} (M) := \{
(f, \eta, \rho) \}/\sim$, i.e., those satisfying the Euler-Lagrange
equations of the Rozansky-Witten action, see
\cite{kapustin,rozansky-witten} or a more general (as per
\cite{qiu-zabzine}) AKSZ formalism below.

The \emph{AKSZ model} \cite{aksz,roytenberg-aksz} uses maps from
$(T[1]M, d_{\DR})$ to $(X,Q)$, where $(T[1]M, d_{\DR})$ is the dg
manifold whose sheaf of functions is the de Rham algebra
$\Omega_M^\bullet : = S^\bullet \TT^*_M [1]$, where $\TT^*_M [1]$ is
the cotangent sheaf of $M$ shifted by degree 1 (also known as the
desuspended cotangent sheaf). The dg structure is given by the de Rham
differential $d_{\DR}$. Also in the above, $(X,Q)$ is a \emph{dg
  symplectic manifold} \cite{roytenberg-aksz}, a dg manifold $(X, Q)$
with a graded skew form $\omega: T_X \otimes T_X \to \R[n-1]$ of
degree $n -1$, where $n = \dim M$, and a Hamiltonian differential $Q$,
i.e., one given by a function $\Theta$ on $X$ of degree $n$ via $Q =
\{\Theta, -\}$, satisfying the "classical master equation" $Q^2 = 0$
or, equivalently, $\{\Theta, \Theta\} = 0$, where the bracket is the
Poisson bracket associated to $\omega$. The classical path space
$\Phi_{\AKSZ} (M)$ will then be defined as the space of dg maps
\[
\Phi_{\AKSZ} (M) := \Map ((T[1]M,d_{\DR}), (X,Q))/\sim
\]
up to the same equivalence relation as in the usual sigma model above.
Since the degree of the symplectic form $\omega$ depends on the
dimension $n$ of the source $M$ and in a TQFT the dimension of $M$
varies, it is not obvious how to generate a TQFT out of the AKSZ
model. We can only speculate that such a TQFT may be given by fixing
the target space to be a graded manifold $X$ with a collection of
symplectic forms $\{\omega_i: T_X \otimes T_X \to \R[i-1] \; | \; i =
0, 1, \dots, n\}$ and differentials $Q_0, Q_1, \dots Q_n$, such that
$Q_i = \{\Theta_i, -\}_i$, where $\{-,-\}_i$ is the Poisson bracket
associated with $\omega_i$ and $\Theta_i$ is a function on $X$ of
degree $i$, satisfying the classical master equation $\{\Theta_i,
\Theta_i\} = 0$. Then for a $n$-cobordism $M$ with corners, the TQFT
functor in the AKSZ model might be defined as the space of dg maps
$\phi: (T[1]M,d_{\DR}) \to (X,Q_n)$ such that restriction of $\phi$ to
a face $F^i$ of $M$ of dimension $i$ induces a dg map
$(T[1]F^i,d_{\DR}) \to (X,Q_i)$. This way, the structure of a span
would be tautological and we would get a TQFT functor, indeed.

\subsection{Variations}

One could consider various modifications of TQFT theory, whose
physical sense may be quite different. For example, one could take
oriented cobordisms with corners. In some cases, such as Yang-Mills
theory, we needed to bring in Riemannian metrics on
cobordisms. Riemannian cobordisms are also needed for Stolz and
Teichner's construction \cite{stolz-teichner} of Segal's elliptic
object. For conformal field theories, one could talk about complex
cobordisms \cite{milnor,ravenel:cc}. It would also make sense to
consider symplectic/contact cobordisms, such as those in
\cite{ginzburg:fa, ginzburg:top}. One can also include D-branes into
the picture by labeling different connected components of the boundary
(or even faces) of cobordisms in selected directions by elements of a
fixed set of ``D-branes.'' For example, in the case of 2-cobordisms
with corners, one can interpret the boundary $\del_1 M$ in direction 1
as closed and open strings, depending on whether the component of the
boundary is closed (circle) or not (interval). The boundary $\del_2 M$
in direction 2 may be regarded as ``free'' boundary and each component
of $\del_2 M$ may be labeled by a D-brane.

Further extension of the notion might be in the spirit of Topological
Conformal Field Theories (TCFTs),
\cite{segal,getzler,me:top,costello:tcfts}, when the source pseudo
$n$-fold category of cobordisms with corners is replaced with a pseudo
$n$-fold category in which a $k$-morphism is a (e.g., singular) chain
of $k$-cobordisms with corners and the target category also has a dg
structure, such as the pseudo $n$-fold category of spans of dg
coalgebras, or cospans of dg algebras (or derived versions
thereof). On top of that, 2d TCFTs cobordisms must also be provided
with extra structure, such as that of a complex manifold. However, we
prefer not to overload this paper with further generalizations in this
direction.

\section{Appendix: the Regular Coherence Theorem}

\begin{thm}[Regular Coherence Theorem]
\label{coherence}
Let $c$ be a \emph{regular} gluing diagram for $k$-morphisms in pseudo
$n$-fold category, that is to say, a diagram made by cutting a
$k$-dimensional cube up by hyperplanes into $p = p_1 \dots p_k$
rectangular blocks of equal size. Let $(k\Mor)^p_c$ denote the full
subcategory of $k\Mor^p$ of morphisms composable in the diagram
$c$. Then any two functors $(k\Mor)^p_c \to k\Mor$ built from morphism
composition are related by a \emph{unique} natural isomorphism built
from associators $\alpha$ and interchangers $\beta$.
\end{thm}

\begin{proof}
  The proof will be based on the following enhancement of Newman's
  Diamond lemma in graph theory. This enhancement seems to be new and
  interesting on its own. Let us start with the classical Diamond
  lemma.

  First of all, recall some terminology related to directed graphs. A
  \emph{directed graph} is a set of vertices and a set of edges
  endowed with a sense of direction, along with an incidence relation
  between the endpoints of the edges and the vertices, so that each
  endpoint of an edge is incident to a vertex. We will consider paths
  between vertices in a graph. A \emph{path} from $x$ to $y$ is given
  by a finite sequence of consecutive undirected edges from $x$ to
  $y$. If all the edges along a path are directed from $x$ to $y$, the
  path is called \emph{directed}. A \emph{descending path} starting at
  a vertex $x$ is an at most countable sequence of consecutive
  directed edges starting at $x$. A path is \emph{trivial} when it
  comprises no edges. A \emph{terminal vertex} is one with no edges
  originating from it.

\begin{lem}[Diamond Lemma]
  Suppose that we are given a directed graph with a \emph{descending
    chain condition} --- namely, that any directed path from a vertex
  has finite length, --- and a \emph{diamond condition} --- namely,
  any two directed edges starting at a vertex may be extended to
  directed paths that end up at the same point. Then every connected
  graph component has a unique terminal vertex.
\end{lem}

For the enhanced version of the lemma, we will need to consider graphs
with an equivalence relation on the set $\Path (x,y)$ of (undirected)
paths from $x$ to $y$ for each pair of vertices $x$ and $y$. We will
assume that this relation $\sim$ satisfies the following conditions:
\begin{description}
\item[trivial cycles] For any path $f$, we have $f^{-1} f \sim \id$,
  where $\id$ is the trivial path at the starting vertex of $f$;
\item[invariance] Two paths from $x$ to $y$ are equivalent, if and
  only if their extensions by an edge adjacent to $x$ (or $y$) are
  equivalent.
\end{description}

\begin{lem}[Enhanced Diamond Lemma]
\label{edl}
Suppose that we have a directed graph with an equivalence relation on
paths, as above. Suppose also that the graph satisfies the descending
chain condition and the following enhanced diamond condition: any two
directed edges originating at a vertex may be extended to
\emph{equivalent} directed paths ending at one and the same
vertex. Then every connected component of the graph has a unique
terminal vertex, and moreover any two paths between any two vertices
are equivalent.
\end{lem}

\begin{proof}[Proof of Lemma]
  We will prove the Enhanced Diamond lemma only, as the proof of
  Newman's classical Diamond lemma is standard and moreover follows
  from our proof of the enhanced version. The proof will follow the
  same sequence of steps as the proof of the classical Diamond lemma,
  see \cite{huet}, based on well-founded (also known as Noetherian)
  induction.
\smallskip

\noindent
\emph{Step 1: Local diamond condition implies global}. First of all,
let us show that under our assumptions, the following ``global''
version of the diamond condition takes place: any two finite directed
\emph{paths} starting at a vertex may be extended to equivalent
directed paths that end up at one and the same vertex.

The proof follows from the \emph{well-founded induction} principle,
which states that if we have a property $P$ of vertices in a directed
graph with a descending chain condition, so that for each vertex $t$,
property$P(t)$ holds whenever it holds for all ends of nontrivial
finite directed paths emanating from $t$, then $P$ holds for all
vertices of the graph.

In our case, the property $P(t)$ is the above global diamond condition
for directed paths originating from a given vertex $t$. The necessary
condition, that is, if $P(t)$ is true for a vertex $t$, then it is
true for all vertices underneath, follows from the invariance
property. To check the sufficient condition, we assume that we have
two directed paths $t \stackrel{*}{\to} u$ and $t \stackrel{*}{\to}
v$, where the asterisk distinguishes a path from an edge, and that the
property $P$ holds for all vertices strictly under $t$; then we have
to find a vertex $w$ with directed paths $u \stackrel{*}{\to} w$ and
$v \stackrel{*}{\to} w$ such that the two composite directed paths
from $t$ to $w$ are equivalent. The cases $t = u$ and $t = v$ are
trivially resolved by setting $w=v$ and $w=u$, respectively.

Otherwise, we may assume the paths from $t$ to $u$ and $v$ to be
nontrivial and thereby write them as $t \to u_1 \stackrel{*}{\to} u$
and $t \to v_1 \stackrel{*}{\to} v$. By the (local) diamond condition,
there exists $w_1$ with directed paths $u_1 \stackrel{*}{\to} w_1$ and
$v_1 \stackrel{*}{\to} w_1$, so that the two composite paths from $t$
to $w_1$ are equivalent. Using the induction hypothesis on $u_1$,
there exists $w_2$ with directed paths $u \stackrel{*}{\to} w_2$ and
$w_1 \stackrel{*}{\to} w_2$, so that the two composite paths from
$u_1$ to $w_2$ are equivalent.

Now we have a composite path $v_1 \stackrel{*}{\to} w_1
\stackrel{*}{\to} w_2$ and the path $v_1 \stackrel{*}{\to} v$,
starting from $v_1$, to which we apply the induction hypothesis again
and get $w$ with paths $w_2 \stackrel{*}{\to} w$ and $v
\stackrel{*}{\to} w$, with the two composite paths from $v_1$ to $w$
being equivalent. In the process, we have obtained three diamonds,
each with equivalent paths from the top to the bottom:
\[
\xymatrix{
& & t \ar[ld] \ar[rd]\\
& u_1 \ar[ld]_{*} \ar@{-->}[rd]^{*} & & v_1 \ar@{-->}[ld]^{*} \ar[rd]^{*}\\
u \ar@{-->}[rd]_{*} & & w_1 \ar@{-->}[ld]^{*} & & v \ar@{-->}[lldd]^{*}\\
& w_2 \ar@{-->}[rd]_{*}\\
& & w}.
\]
The big, outer diamond is the one we have been looking for, and the
paths down along its edges are equivalent, because the smaller
diamonds are bounded by equivalent paths.
\smallskip

\noindent
\emph{Step 2: The uniqueness of a normal form}. This means that for
any vertex $x$, there exists a unique terminal vertex $y$ with a
directed path from $x$ to $y$, unique up to equivalence. This $y$ is
called the \emph{normal form of} $x$. The existence is obvious from
the descending chain condition. The uniqueness of a normal form and a
directed path to it, up to equivalence, follows trivially from the
enhanced global diamond property.
\smallskip

\noindent
\emph{Step 3: If two vertices are connected by a path, then their
  normal forms are the same and the path between the two vertices is
  unique up to equivalence}. Note that the two vertices may be the
same, in which case the uniqueness of a path up to equivalence is
still a nontrivial statement. It implies that any closed path is
equivalent to the trivial path.

The exact statement we are going to prove will be more constructive:
\emph{if two vertices $x$ and $y$ are connected by a path, then they
  have one and the same normal form $z$ and the path is equivalent to
  a path $x \stackrel{*}{\to} z \stackrel{*}{\leftarrow} y$}.

This statement may be proven using usual induction on the length of
the path. If the path is trivial, the statement follows from Step 2
and the trivial cycles property. Suppose the statement is true for all
paths of length at most $n \ge 0$. If we have a path from $x$ to $y$
consisting of $n+1$ edges, consider the vertex $x'$ one edge away from
$x$ in the direction of $y$ along the path. By induction, $x'$ and
$y$, connected by a path of length $n$, have the same normal form $z$
and the path from $x'$ to $y$ is equivalent to $x' \stackrel{*}{\to} z
\stackrel{*}{\leftarrow} y$. Whichever way the edge between the
vertices $x$ and $x'$ is directed, a normal form of one of them will
automatically be a normal form of the other. For instance, if the edge
is directed from $x'$ to $x$, a directed path from to $x$ to a normal
form extended by the edge from $x'$ to $x$ will serve as a directed
path from $x'$ to a normal form thereof. Because of the uniqueness of
a normal form, it will be the vertex $z$. Because of the uniqueness of
a path to a normal form, the path $x' \xrightarrow{*} z$ will be
equivalent to the path $x' \to x \xrightarrow{*} z$.

Note that the path $x \leftarrow x'$ is equivalent to $x
\stackrel{*}{\to} z \stackrel{*}{\leftarrow} x \leftarrow x'$. Thus,
we see that the path from $x$ to $y$ is equivalent to the path $x
\stackrel{*}{\to} z \stackrel{*}{\leftarrow} x' \stackrel{*}{\to} z
\stackrel{*}{\leftarrow} y$, in which the two middle arrows may be
``canceled,'' so that it becomes equivalent to the path $x
\stackrel{*}{\to} z \stackrel{*}{\leftarrow} y$. The case when the
edge between $x$ and $x'$ is directed from $x$ to $x'$ is treated
similarly.
\end{proof}

Now we are ready to prove the Regular Coherence theorem. The existence
statement is a particular case of the Weak Coherence theorem,
Theorem~\ref{existence}, and all we need is to prove uniqueness.

We would like to apply the Enhanced Diamond Lemma, in which the graph
is formed by functors $(k\Mor)^p_c \to k\Mor$ built from morphism
composition as vertices and applicable natural transformations
$\alpha_i$ and $\beta_{ji}$ with $i < j$ (combined with identities in
all the remaining variables) as directed edges. We say that two paths
in this graph are equivalent, if the resulting compositions of natural
transformations are equal. This is an equivalence relation satisfying
the trivial cycles and invariance conditions, see the discussion
before Lemma~\ref{edl}.

Let us check that this graph satisfies the descending chain
condition. Note that each associator $\alpha$ moves a pair of
parentheses to the right. Thus, in a descending path of $\alpha$s and
$\beta$s, there should be a finite number of $\alpha$s. Note that each
interchanger $\beta_{ji}$ with $i < j$ leaves the parentheses intact,
but moves morphism composition $\circ_j$ with greater index $j$ from
the inside of a pair of parentheses to the outside. Even though some
of the $\alpha$s participating in a path may move compositions
$\circ_j$ with greater indices $j$ deeper inside parentheses, there
are only finitely many $\alpha$s and thereby there should be only
finitely many $\beta$s.

Thus, to show the uniqueness in the Regular Coherence theorem, we just
need to verify the enhanced diamond condition. This is far from being
obvious and is the core argument in the proof of the Regular Coherence
theorem.

The enhanced diamond condition starts with two directed edges
emanating from a vertex. This translates into two natural
transformations $\gamma$ and $delta$, each of the type $\alpha_i$ or
$\beta_{ji}$ with $i < j$, applied to a composition of some number $p$
of $k$-morphisms, or strictly speaking, a functor $(k\Mor)^p_c \to
k\Mor$ built out of morphism compositions. We need to show that these
transformations may be appended by transformations of same kind, so
that the resulting sequence of transformations will end at one and the
same composition functor $(k\Mor)^p_c \to k\Mor$ and the corresponding
paths will be equivalent.

To check the enhanced diamond condition, we will use induction on the
number $p$ of terms in our functors $(k\Mor)^p_c \to k\Mor$, starting
with $p=2$, when the gluing diagram $c$ consists of two blocks
attached in a particular direction $i$, $1 \le i \le n$, and the only
composition functor applied to two morphisms fitting into this
diagram: $\circ_i$. In this case, the graph consists of just a single
vertex, and the enhanced diamond condition is vacuously true. Suppose
we have proven the statement for any number less than $p$ of terms in
our functors. We need to check if the statement is true for $p$ terms.

Note that the induction assumption implies the existence of a normal
form for any regular composition of less than $p$ morphisms. It is
easy to describe what this normal form has to be. Without loss of
generality, let us assume that morphisms in the composition are
composed in directions $1, \dots, d$, for some $d$: $1 \le d \le
k$. Then in a normal form, firstly, all compositions in direction 1
are made, with all the parentheses moved to the right, as in $a_1
\circ_1 (a_2 \circ_1 (\dots \circ_1 (a_{q-1} \circ_1 a_q)\dots))$,
then all compositions in direction 2 are made, again with all the
parentheses moved to the right, etc., the last compositions are made
in direction $d$, with all parentheses on the right. This is a normal
form, because in our directed graph it represents a terminal vertex:
no \emph{forward moves}, i.e., moves along directed edges of the
graph, $\alpha_i$ or $\beta_{ji}$ with $i < j$, can be applied to this
composition.

Next, observe that our functor $(k\Mor)^p_c \to k\Mor$ is a
composition two such functors, say $a \circ_i b$, each with a strictly
less than $p$ number of terms. One has to consider a number of cases,
depending on what the natural transformations $\gamma$ and $\delta$
are. There are the following possibilities for $\gamma$:
\begin{enumerate}
\item $\gamma$ acts inside $a$;
\item $\gamma$ acts inside $b$;
\item \label{case3} $a= a_1 \circ_i a_2$, and $\gamma = \alpha_i: (a_1
  \circ_i a_2) \circ_i b \to a_1 \circ_i (a_2 \circ_i b)$ is an
  associator;
\item \label{case4} $a= a_1 \circ_j a_2$, $b= b_1 \circ_j b_2$, and
  $\gamma = \beta_{ji}: (a_1 \circ_j a_2) \circ_i (b_1 \circ_j b_2)
  \to (a_1 \circ_i b_1) \circ_j (a_2 \circ_i b_2)$ with $i < j$ is an
  interchanger.
\end{enumerate}
The same possibilities are there for $\delta$. If they act both within
$a$ or both within $b$, we are done by the induction assumption. If
$\gamma$ acts within one term and $\delta$ within the other, then they
commute because of the functoriality of $\circ_i$, and we get a
commutative square. If both $\gamma$ and $\delta$ fall under Cases
\eqref{case3} and \eqref{case4}, then the only possibility for this to
happen is when $\gamma = \delta$, and the enhanced diamond condition
holds trivially. Thus, what remains to be done is to check the
enhanced diamond condition when $\gamma$ is an associator or
interchanger and $\delta$ acts within either $a$ or $b$, or vice
versa.

Suppose $\gamma = \alpha_i$ is an associator, as in
Case~\eqref{case3}. If $\delta$ works entirely within $a_1$, $a_2$, or
$b$, then $\gamma$ and $\delta$ commute by the naturality of
$\alpha_i$. The only remaining cases here are (1) $a_1 = a_{11}
\circ_i a_{12}$ and $\delta$ is the associator $((a_{11} \circ_i
a_{12}) \circ_i a_2) \circ_i b \to (a_{11} \circ_i (a_{12} \circ_i
a_2)) \circ_i b$, and (2) when $a_1 = a_{11} \circ_j a_{12}$ and $a_2
= a_{21} \circ_j a_{22}$ for some $j > i$ and $\delta$ is the
interchanger $\beta_{ji}: ((a_{11} \circ_j a_{12}) \circ_i (a_{21}
\circ_j a_{22})) \to (a_{11} \circ_i a_{21}) \circ_j (a_{12} \circ_i
a_{22})$ appended with $\circ_i b$. In the first case, the two
associators $\gamma$ and $\delta$ are the first two edges in the
coherence pentagon, and we can continue them as paths around the
pentagon to the opposite vertex. The resulting compositions of the
natural transformations will be equal, because of the pentagon
axiom. The above argument repeats a standard argument in the proof of
Mac Lane's coherence theorem for monoidal categories, see
\cite{unapologetic} or \cite{maclane}.

In the second case, we have $\gamma = \alpha_i: (a_1 \circ_i a_2)
\circ_i b \to a_1 \circ_i (a_2 \circ_i b)$ and $\delta = \beta_{ji}
\circ_i \id: ((a_{11} \circ_j a_{12}) \circ_i (a_{21} \circ_j a_{22}))
\circ_i b \to ((a_{11} \circ_i a_{21}) \circ_j (a_{12} \circ_i
a_{22})) \circ_i b$, where $a_1 = a_{11} \circ_j a_{12}$ and $a_2 =
a_{21} \circ_j a_{22}$, and need to check the enhanced diamond
condition for these two edges. If $b$ were factored as $b_1 \circ_j
b_2$, with $b_1$ and $b_2$ fitting in direction $i$ with $a_{21}$ and
$a_{22}$, respectively, we would then proceed as in the previous case
and just use the first hexagon axiom. If the composition diagram $c$
were not regular, $b$ would not factor like that in general, even
after applying a few associators and interchangers. Given that our
composition diagram is regular, we could use coherence
transformations, existing by the induction assumption, to rearrange
how $b$ is composed and indeed factor it as $b_1 \circ_i b_2$ to fit
with $a_{21}$ and $a_{22}$, as above, but the problem is that it would
not help to check the diamond condition, as this rearrangement would
not necessarily be achieved by forward moves. To deal with this
problem, we will use the induction assumption to put each factor
$a_{11}, a_{12}, a_{21}, a_{22}, b$ in a normal form by forward
moves. In particular, since we assumed that compositions in our
composition diagram $c$ are performed in directions $1, 2, \dots, d$,
a normal form of the above five terms will always be a composition in
direction $d$: $ - \circ_d -$, and we will also have $i < j \le d$. We
will have to consider two cases: $j = d$ and $j < d$. If $j=d$, we
will only need to put $a_{11}$, $a_{21}$, and $b$ in a normal form:
\begin{eqnarray*}
  a_{11} & = & a_{111} \circ_j a_{112},\\
  a_{21} & = &a_{211} \circ_j a_{212},\\
  b &= & b_1 \circ_j b_2.
\end{eqnarray*}
Because of the regularity of our initial diagram $c$, morphisms in it
will fit each other as follows:
\[
\xymatrix{
  \bullet \ar[r] \ar[d] \ar @{} [dr] |{a_{111}} & \bullet \ar[r] \ar[d] \ar @{} [dr] |{a_{211}} & \bullet \ar[r] \ar[d] \ar @{} [dr] |{b_1}  & \bullet \ar[d]\\
  \bullet \ar[r] \ar[d] \ar @{} [dr]  |{a_{112}} & \bullet \ar[r] \ar[d] \ar @{} [dr]  |{a_{212}} & \bullet \ar[r] \ar[d] \ar @{} [ddr]  |{b_2}  & \bullet \ar[d]\\
  \bullet \ar[r] \ar[d] \ar @{} [dr]  |{a_{12}} & \bullet \ar[r] \ar[d] \ar @{} [dr]  |{a_{22}} & \bullet \ar[d]  & \bullet \ar[d]\\
  \bullet \ar[r] &\bullet \ar[r] &\bullet \ar[r] & \bullet}
\]
Then we will apply to $(((a_{111} \circ_j a_{112}) \circ_j a_{12})
\circ_i ((a_{211} \circ_j a_{212}) \circ_j a_{22})) \circ_i (b_1
\circ_j b_2)$, as per this diagram, several successive forward moves
with the idea of splitting off the top row $a_{111} \circ_i (a_{211}
\circ_i b_1)$ and using induction on the remaining smaller
diagram. This process may be described by the following scheme, see
Figure~\ref{typical}: we will extend the edges $\gamma$ and $\delta$,
which start at vertex 1, to directed paths ending at vertices 3 and 4,
respectively. We will construct directed paths from vertex 1 to these
vertices via a new, common vertex 2 in such a way that these new paths
are equivalent to the respective previous ones via a sequence of
coherence diagrams. The paths from 2 to 3 and 2 to 4 will actually be
directed edges of the type $\id \circ_d \gamma'$ and $\id \circ_d
\delta'$ (remember that $d=j$), and those may be extended to two
equivalent paths to a common vertex 5 by the induction assumption,
applied to $\gamma'$ and $\delta'$.

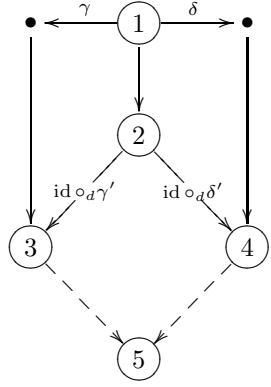
\begin{figure}
\[
\xymatrix{ \bullet \ar[dd] & *++[o][F-]{1} \ar[l]_{\gamma}
  \ar[r]^{\delta} \ar[d]
  & \bullet  \ar[dd]\\
  & *++[o][F-]{2} \ar[dl]|{\id \circ_d \gamma'} \ar[dr]|{\id \circ_d \delta'}\\
  *++[o][F-]{3} \ar@{-->}[dr] & & *++[o][F-]{4} \ar@{-->}[dl]\\
  & *++[o][F-]{5} }
\]
\caption{Inductive Argument Scheme}
\label{typical}
\end{figure}

In our case, when $\gamma = \alpha_i$ and $\delta = \beta_{ji} \circ_i
\id$, this scheme is realized in the diagram in
Figure~\ref{alphabeta}.

\begin{sidewaysfigure}
\vspace{5in}
\hspace{0.5in}
\includegraphics[height=8in,angle=90]{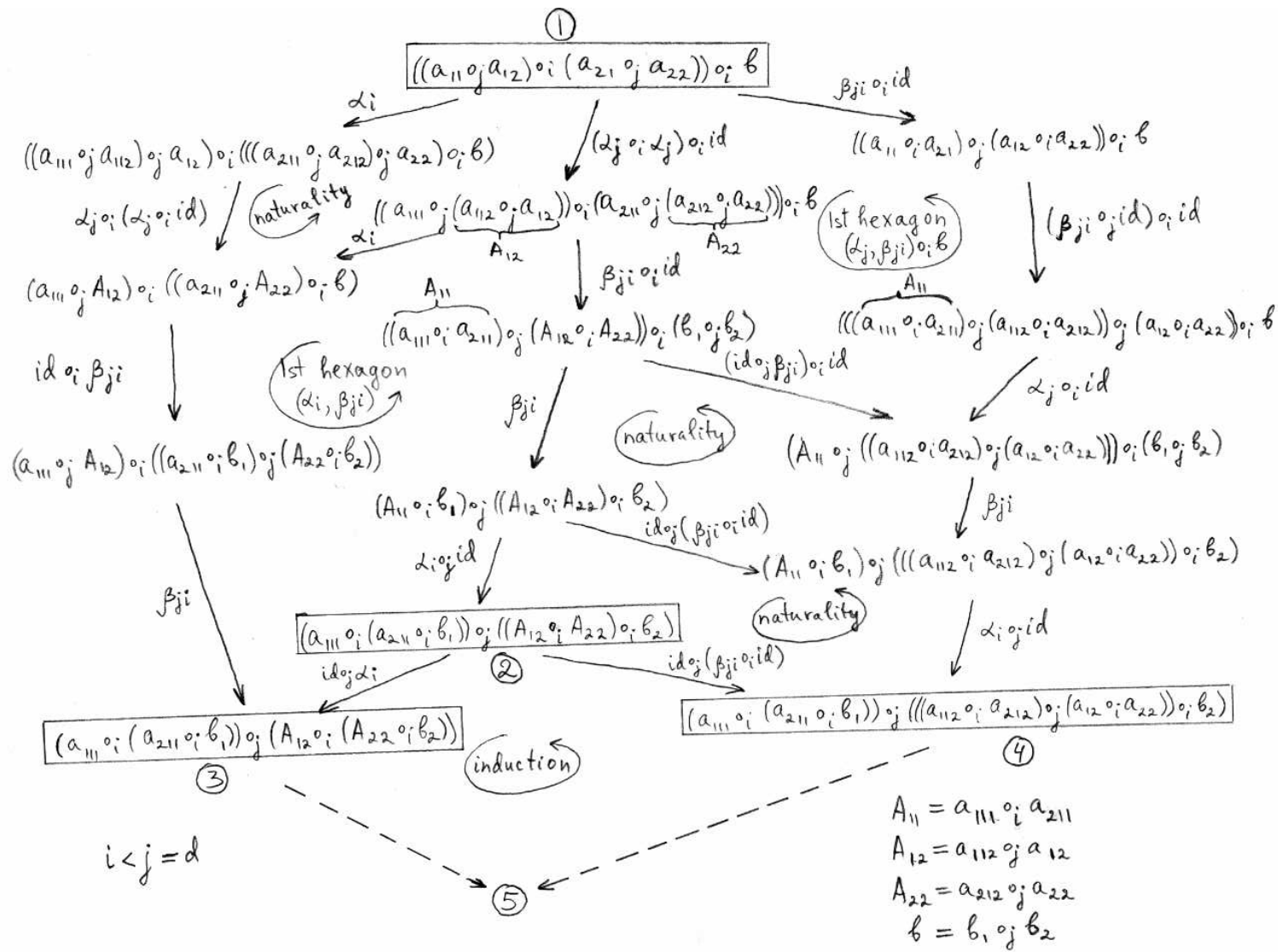}
\caption{Verification of the Diamond Condition: $\alpha$ and $\beta$}
\label{alphabeta}
\end{sidewaysfigure}

This completes considering the case when $\gamma = \alpha_i$ and
$\delta = \beta_{ji}$ with $j = d$.

In principle, the same case with $j<d$ and the remaining
Case~\eqref{case4}, when $\gamma = \beta_{ji}$ and $\delta$ acts
entirely within $a$ or $b$, or vice versa, are done similarly,
following the general idea of Figure~\ref{typical}, except that there
are extra dimensions involved, which is something new as compared to
classical coherence theorems for monoidal categories. We will present
the case which we found most complicated and leave the other cases as
an exercise for the reader.

In this case, we will use the diagram in Figure~\ref{diamond}, in
which $\delta = \beta_{ji}: (a_1 \circ_j a_2) \circ_i (b_1 \circ_j
b_2) \to (a_1 \circ_i b_1) \circ_j (a_2 \circ_i b_2)$ with $i < j$ and
$\gamma$ acts within $a_1 \circ_j a_2$ as $\gamma = \beta_{kj} \circ_i
\id: ((a_{11} \circ_k a_{12}) \circ_j (a_{21} \circ_k a_{22})) \circ_i
(b_1 \circ_j b_2) \to ((a_{11} \circ_j a_{21}) \circ_k (a_{12} \circ_j
a_{22})) \circ_i (b_1 \circ_j b_2)$ with $j < k < d$. Again, the
general scheme may be described by Figure~\ref{betas}.

\begin{sidewaysfigure}
\vspace{5in}
\hspace{-1in}
\includegraphics[width=9in]{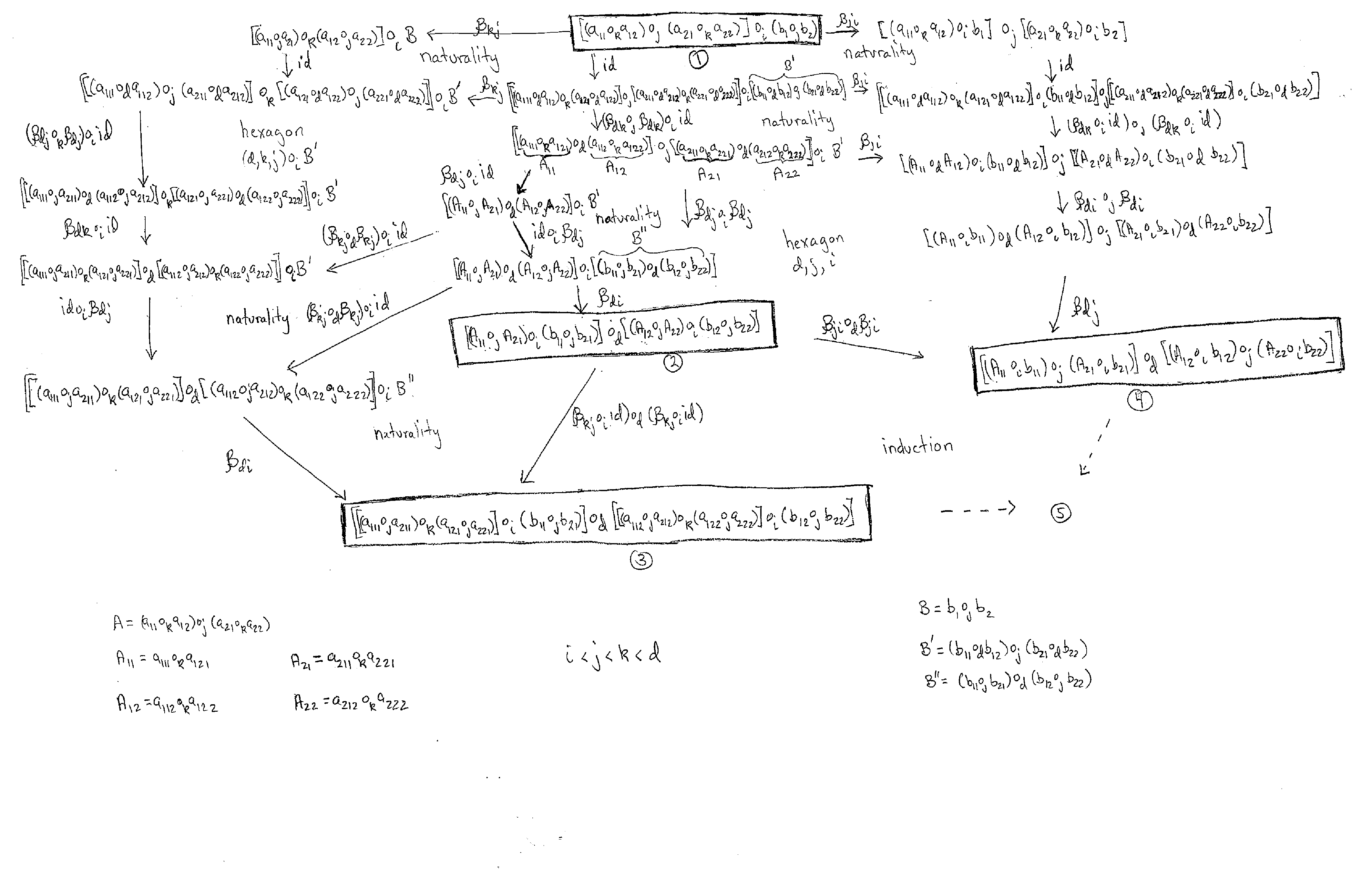}
\caption{Verification of the Diamond Condition: Two $\beta$s}
\label{diamond}
\end{sidewaysfigure}

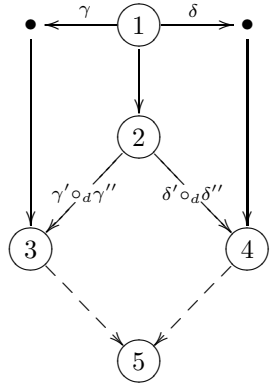
\begin{figure}
\[
\xymatrix{
\bullet \ar[dd]
& *++[o][F-]{1} \ar[l]_{\gamma} \ar[r]^{\delta} \ar[d]
& \bullet  \ar[dd]\\
& *++[o][F-]{2} \ar[dl]|{\gamma' \circ_d \gamma''} \ar[dr]|{\delta' \circ_d \delta''}\\
*++[o][F-]{3} \ar@{-->}[dr] & & *++[o][F-]{4} \ar@{-->}[dl]\\
& *++[o][F-]{5}
}
\]
\caption{General Scheme of Diagram in Figure~\ref{diamond}}
\label{betas}
\end{figure}

\end{proof}

\bibliographystyle{amsalpha}
\bibliography{ncat}
\end{document}